\documentclass{sinst}

\usepackage{amsbsy,amsfonts,amsmath,amssymb,amsthm}
\usepackage{array}
\usepackage{bm}
\usepackage{color}
\usepackage{enumerate}
\usepackage{mathrsfs}
\usepackage{latexsym}
\usepackage[colorlinks=true]{hyperref}
\usepackage{mathtools}
\usepackage{cleveref}
\hypersetup{urlcolor=blue, citecolor=blue, linkcolor=blue}
\definecolor{wineRed}{rgb}{0.7,0,0.3}
\definecolor{grandBleu}{rgb}{0,0,0.8}
\definecolor{darkGreen}{rgb}{0,0.4,0}
\definecolor{blueViolet}{rgb}{0.4,0,1.0}
\definecolor{bloodOrange}{rgb}{0.85,0.05,0}
\definecolor{mycolor}{rgb}{0.8,0,0.2}
\usepackage[textsize=small]{todonotes}
\usepackage{cite}

\usepackage{comment}

\DeclareMathAlphabet{\mathpzc}{OT1}{pzc}{m}{it}

\usepackage[textsize=small]{todonotes}
\numberwithin{equation}{section}

\theoremstyle{plain}
\newtheorem{mainThm}{Main Theorem} 
\newtheorem{lemma}{Lemma}[section]

\theoremstyle{definition}
\newtheorem{definition}{Definition}%[section]
\newtheorem{theorem}{Theorem}%[section]

\def\F{\mathcal{F}}
\def\N{\mathbb{N}}
\def\R{\mathbb{R}}

\def\L{\mathcal{L}}

\def\sH{\mathscr{H}}
\def\cT{\mathcal{T}}
\def\sV{\mathscr{V}}
\def\ds{\displaystyle}

\def\Lap{\mathit{\Delta}}

\DeclareMathOperator{\diver}{div}

\begin{document}
\vspace*{-0cm}
\title{Well-posedness of a Pseudo-Parabolic {K}{W}{C} System \\ in Materials Science
\footnotemark[1]
\vspace{-2ex}
}
%%
%\dedication{A line for dedication}
%%
\author{Harbir Antil
%\footnotemark[1]
}
\affiliation{Department of Mathematical Sciences and \\ the Center for Mathematics and Artificial Intelligence,  (CMAI), 
\\
George Mason University, Fairfax, VA 22030, USA}
\email{hantil@gmu.edu}
\vspace{-2ex}

\sauthor{Daiki Mizuno
%\footnotemark[2]
}
\saffiliation{Division of Mathematics and Informatics, \\ Graduate School of Science and Engineering, Chiba University, \\ 1-33, Yayoi-cho, Inage-ku, 263-8522, Chiba, Japan}
\semail{d-mizuno@chiba-u.jp}
\vspace{-2ex}

\tauthor{Ken Shirakawa
%\footnotemark[3]
}
\taffiliation{Department of Mathematics, Faculty of Education, Chiba University \\ 1--33 Yayoi-cho, Inage-ku, 263--8522, Chiba, Japan}
\temail{sirakawa@faculty.chiba-u.jp}
\vspace{-2ex}

%\tauthor{
%\footnotemark[4]
%}
%\taffiliation{}
%\temail{sirakawa@faculty.chiba-u.jp}
%\vspace{-10ex}
%\fauthor{Name of 4th author}
%\faffiliation{1st line of affiliation \\ 2nd line of affiliation}
%\femail{e-mail address of 4th author}
%Please do not use \date 
\footcomment{
$^*$\,This work is supported by  Grant-in-Aid for Scientific Research (C) No. 20K03672, JSPS.
HA and KS are also partially supported by NSF grant DMS-2110263, the Air Force Office
of Scientific Research (AFOSR) under Award NO: FA9550-22-1-0248, and the 
Office of Naval Research (ONR) under Award NO: N00014-24-1-2147.\\
AMS Subject Classification: 
35A15, %Variational methods applied to PDEs
35G50, % Systems of nonlinear higher-order PDEs
35G61, %Initial-boundary value problems for systems of nonlinear higher-order PDEs
%35K67,  % Singular parabolic equations 
35K70, % Ultraparabolic equations, pseudoparabolic equations, etc.
%49J20, %Existence theories for optimal control problems involving partial differential equations
%49K20. % Optimality conditions for problems involving partial differential equations
%74N05, % Crystals in solids
74N20. % Dynamics of phase boundaries in solids
\\
Keywords: KWC-type system, pseudo-parabolic nature, existence, uniqueness, regularity, continuous-dependence
}
%\footnotetext[1]{aaa}
%\footnotetext[2]{bbb}
\maketitle

\noindent
{\bf Abstract.}
The original KWC-system is widely used in materials science. It was proposed in [Kobayashi et al, Physica D, 140, 141--150 (2000)] and is based on the phase field model of planar grain boundary motion. This model suffers from two key challenges. Firstly, it is difficult to establish its relation to physics, in particular, a variational model. Secondly, it lacks uniqueness. The former has been recently studied  within the realm of BV-theory. The latter only holds under various simplifications. This article introduces a pseudo-parabolic version of the KWC-system. A direct relationship with variational model (as gradient-flow) and uniqueness are established without making any unrealistic simplifications. Namely, this is the first KWC-system which is both physically and mathematically valid. The proposed model overcomes the well-known open issues.
\newpage

\section{Introduction}

The Kobayashi--Warren--Carter (KWC) system consists of a set of non-smooth parabolic PDEs and is widely used in materials science \cite{MR1752970,MR1794359}. It is based on the phase field model of planar grain boundary motion. This model suffers from two fundamental challenges: 
(i) \emph{Physics:} It is difficult to establish KWC-system as the gradient flow of a variational model; 
(ii) \emph{Mathematics:} The solution to KWC-system are known to be unique only under special 
cases. Both of these challenges makes it difficult to rigorously use this model in practice or carry 
our new material design via optimization \cite{MR4218112,MR4395725}. 

This article aims to overcome both these challenges by introducing a pseudo-parabolic version 
of the KWC-system. Well-posedness (both existence and uniqueness) of the resulting system (S),
which arise from gradient flow based on the \emph{KWC-energy}
\begin{gather}
    \mathcal{F} : [\eta, \theta] \in [L^2(\Omega)]^2 \mapsto \mathcal{F}(\eta, \theta) := \frac{1}{2} \int_\Omega |\nabla \eta|^2 \, dx +\int_\Omega G(\eta) \, dx 
    \label{FE}
    \\
    +\int_\Omega \alpha(\eta) |D \theta| \in [0, \infty],
    \nonumber
\end{gather}
is established. Namely, this article addresses open issues from previous works that deal with the KWC-system (cf. \cite{MR3268865,MR3670006,MR3038131,
%MR1857292,MR3409135,MR0775682,MR3288271,MR1259102,MR2306643,
MR4352617,nakayashiki2023kobayashi,
kubota2023periodic}) and its regularized versions (cf. \cite{MR2469586,MR3155454,
MR3888633,MR4218112,MR4395725,MR4228007}). 

Next, we describe the system (S). Let $ 0 < T < \infty $ be a fixed final time, and let $ N \in \{1, 2, 3\} $ 
denote the spatial dimension. Let $ \Omega \subset \R^N $ be an open bounded spatial domain with a boundary $ \Gamma := \partial \Omega $. When $ N > 1 $, $ \Gamma $ is assumed to be sufficiently 
smooth, with the unit outer normal $ n_\Gamma $. Besides, we let $ Q := (0, T) \times \Omega $ and 
$ \Sigma := (0, T) \times \Gamma $.  
Then the pseudo-parabolic system denoted by (S), with two constants $ \mu \geq 0 $ and $ \nu > 0 $, 
is given by 
%\begin{subequations}\label{(S)}
\begin{align}
    \mbox{(S)}: ~~&
    \nonumber
    \\
    & \begin{cases}
        \partial_t \eta -\mathit{\Delta} \bigl( \eta +\mu^2 \partial_t \eta \bigr) +g(\eta) +\alpha'(\eta(t))|\nabla \theta| = u(t, x), ~ \mbox{for $ (t, x) \in Q $,}
        \\[0.5ex]
        \nabla \bigl( \eta +\mu^2 \partial_t \eta \bigr) \cdot n_\Gamma =  0, ~ \mbox{on $ \Sigma $,}
        \
        \\[0.5ex]
        \eta(0, x) = \eta_0(x),~ \mbox{for $ x \in \Omega $,}
    \end{cases}
    \nonumber
    %\label{(S)1st}
    \\[1.5ex]
    & \begin{cases}
        \ds \alpha_0(\eta) \partial_t \theta -\mathrm{div} \left( \alpha(\eta) \frac{D \theta}{|D \theta|} +\nu^2 \nabla \partial_t \theta \right) = v(t, x), ~ \mbox{for $ (t, x) \in Q $,}
        \\[1ex]
        \bigl( \alpha(\eta) \frac{D \theta}{|D \theta|} +\nu^2 \nabla \partial_t \theta \bigr) \cdot n_\Gamma = 0 ~ \mbox{on $ \Sigma $,}
        \\[0.5ex]
        \theta(0, x) = \theta_0(x),~ \mbox{for $ x \in \Omega $.}
    \end{cases}
    \nonumber
    %\label{(S)2nd}
\end{align} 
%\end{subequations}
Here, the unknowns $ \eta = \eta(t, x) $ and $ \theta = \theta(t, x) $ are order parameters that indicate the \emph{orientation order} and \emph{orientation angle} of the polycrystal body, respectively. Besides, $ \eta_0 = \eta_0(x) $ and $ \theta_0 = \theta_0(x) $ is the \emph{initial data}. Moreover, $ u = u(t, x) $ and $ v = v(t, x) $ are the forcing terms. Additionally, $ \alpha_0 = \alpha_0(\eta) $ and $ \alpha = \alpha(\eta) $ are fixed positive-valued functions to reproduce the mobilities of grain boundary motions. Finally, $ g = g(\eta) $ is a perturbation for the orientation order $ \eta $, having a nonnegative potential $ G = G(\eta) $, i.e. $ \frac{d}{d\eta} G(\eta) = g(\eta) $. 

A generic form of the ``KWC-system'' is given by the evolution equation (cf. \cite{MR1752970,MR1794359}):
\begin{gather}\label{EE}
    -\mathcal{A}_0\bigl(\eta(t) \bigr) \frac{d}{dt} \left[ \begin{matrix}
        \eta(t) \\[1ex] \theta(t)
    \end{matrix} \right] = \frac{\delta}{\delta [\eta, \theta]} \mathcal{F}(\eta(t), \theta(t)) +\left[ \begin{matrix}
        u(t) \\[1ex] v(t)
    \end{matrix}
    \right]
    \\[1ex]
    \mbox{in $ [L^2(\Omega)]^2 $, for $ t \in (0, T) $,}
    \nonumber
\end{gather}
which is motivated by the gradient flow of the free-energy, namely the KWC-energy \eqref{FE}, 
with a functional derivative $ \frac{\delta}{\delta [\eta, \theta]} \mathcal{F} $, and an unknown-dependent monotone operator $ \mathcal{A}_0(\eta) \subset [L^2(\Omega)]^2 $. Here, the evolution equation \eqref{EE} can be considered as the common root of the original KWC-system (cf. \cite{MR1752970,MR1794359}) and our system (S). Indeed, our system (S) is derived from the evolution equation \eqref{EE}, in the case when:
\begin{gather}
    \label{A_mu^nu}
    \mathcal{A}_0(\eta) : \left[ \begin{matrix}
        \widetilde{\eta} \\[1ex] \widetilde{\theta}
    \end{matrix} \right] \in 
    [H^2(\Omega)]^2 
    \mapsto \mathcal{A}_0(\eta) \left[ \begin{matrix}
        \widetilde{\eta} \\[1ex] \widetilde{\theta}
    \end{matrix} \right] :=  
    \left[ \begin{matrix}
        \widetilde{\eta} -\mu^2 \mathit{\Delta} \widetilde{\eta} \\[1ex] \alpha_0(\eta) \widetilde{\theta} -\nu^2 \mathit{\Delta} \widetilde{\theta}
    \end{matrix} \right] \in 
    [L^2(\Omega)]^2,
    \\
    \nonumber
    \mbox{for each $ \eta \in L^2(\Omega) $, subject to the zero-Neumann boundary condition,}
\end{gather}
while the original KWC-system corresponds to the case when $ \mu = \nu = 0 $. 

In recent years, the principal issue has been to clarify the variational structure (representation) of the functional derivative $ \frac{\delta}{\delta [\eta, \theta]} \mathcal{F} $ of the nonsmooth and nonconvex energy $ \mathcal{F}$ in \eqref{FE}. The positive answer to the issue was obtained in \cite{MR3268865,MR3670006,MR3038131}, by means of BV-theory (cf. \cite{MR1857292,MR3409135,MR0775682,MR3288271,MR1259102,MR2306643}), and this work has provided a basis of the study of KWC-system, e.g. the existence and large-time behavior \cite{MR3268865,MR3670006}, the observations under other boundary conditions \cite{MR4352617,nakayashiki2023kobayashi}, the time-periodic solution \cite{kubota2023periodic}, and so on. 

    However, the uniqueness of solutions has been a significant challenge, due to the velocity term $\alpha_0(\eta) \partial_t \theta $ and the singular diffusion flux $ \alpha(\eta) \frac{D \theta}{|D \theta|} $, both of which depend on the unknown-dependent mobilities. Therefore, previous researchers have implemented the following modifications to the modelling framework \eqref{FE} and \eqref{A_mu^nu}:
    \vspace{-1ex}
\begin{description}
    \item[$\bullet$] resetting $ \alpha_0 $ to be a function which is independent of $ \eta $ (effectively a constant);
    \vspace{1ex}
    \item[$\bullet$] modifying the free-energy functional to a more relaxed form:
\end{description}
\begin{gather*}%\label{FEeps}
\mathcal{F}_\varepsilon : [\eta, \theta] \in [L^2(\Omega)]^2 \mapsto \mathcal{F}_\varepsilon(\eta, \theta) = \mathcal{F}(\eta, \theta) + \frac{\varepsilon^2}{2} \int_\Omega |\nabla \theta|^2 \, dx, 
    \\
    \text{with a small constant $ \varepsilon > 0 $.}
    \nonumber
\end{gather*}
These modifications have been pivotal in addressing the uniqueness challenges (cf. \cite{MR2469586,MR3155454}), and several advanced issues, such as the optimal control problems (see \cite{MR3888633,MR4218112,MR4395725,MR4228007}). 

In light of this, we can expect that the pseudo-parabolic nature of our system will effectively address the uniqueness challenge. This is due to the positive constants $ \mu $ and $ \nu $ in \eqref{A_mu^nu}, which are expected to bring a smoothing effect for the regularity of solution. In addition, it is also crucial from a mathematical perspective to clarify the similarities and differences between our pseudo-parabolic system and the original parabolic KWC-system. 

Consequently, we set the goal to prove the following two Main Theorems, concerned with the well-posedness of our pseudo-parabolic system (S), i.e. the evolution equation \eqref{EE} under \eqref{FE} and \eqref{A_mu^nu}. 
\vspace{2ex}
\begin{description}
    \item[Main Theorem 1.]Existence and regularity of solution to (S).
\vspace{1ex}
    \item[Main Theorem 2.]Uniqueness of solution to (S), and continuous dependence with respect to the initial data and forcings. 
\end{description}
\vspace{2ex}

    These Main Theorems will provide the positive answer to our earlier expectation regarding the effectiveness of the pseudo-parabolic nature of our system in resolving the uniqueness issue. Also, the Main Theorems will focus on two conflicting properties: the singularity in the diffusion flux $ \alpha(\eta)\frac{D \theta}{|D \theta|} $; and the smoothing effect encouraged by the Laplacian in \eqref{A_mu^nu}. This conflicting situation will be clarified by the differences in regularity between components: $ \eta \in W^{1, 2}(0, T; H^2(\Omega)) $; and $ \theta \in W^{1, 2}(0, T; H^1(\Omega)) \cap L^\infty(0, T; H^2(\Omega)) $; in the Main Theorem 1. Moreover, the results of this paper will form a fundamental part of the optimization problem in grain boundary motion, which will be exploring in a forthcoming paper. 

\medskip
\noindent 
{\bf Outline:} 
Preliminaries are given in Section 1, and on this basis, the Main Theorems are stated in Section 2. For the proofs of Main Theorems, we prepare Section 3 to set up an approximation method for (S). Based on these, the Main Theorems are proved in Section 4, by means of the auxiliary results obtained in Section 3. 

\section{Preliminaries}
\label{sec:Prelim}
We begin by prescribing the notations used throughout this paper. 
\bigskip

\noindent
\underline{\textbf{\textit{Notations in real analysis.}}}
We define:
\begin{align*}
    & r \vee s := \max \{ r, s \} ~ \mbox{ and } ~ r \wedge s := \min \{r, s\}, \mbox{ for all $ r, s \in [-\infty, \infty] $,}
\end{align*}
and especially, we write:
\begin{align*}
    & [r]^+ := r \vee 0 ~ \mbox{ and } ~ [r]^- := -(r \wedge 0), \mbox{ for all $ r \in [-\infty, \infty] $.}
\end{align*}
Additionally, for any $M > 0$, let $\cT_M: \R \longrightarrow [-M,M]$ be the truncation operator, defined as:
\begin{equation*}
    \cT_M: r \in \R \mapsto (r \vee (-M)) \wedge M \in [-M,M].
\end{equation*}

Let $ d \in \N $ be a fixed dimension. We denote by $ |y| $ and $ y \cdot z $ the Euclidean norm of $ y \in \mathbb{R}^d $ and the scalar product of $ y, z \in \R^d $, respectively, i.e., 
\begin{equation*}
\begin{array}{c}
| y | := \sqrt{y_1^2 +\cdots +y_d^2} \mbox{ \ and \ } y \cdot z  := y_1 z_1 +\cdots +y_d z_d, 
\\[1ex]
\mbox{ for all $ y = [y_1, \ldots, y_d], ~ z = [z_1, \ldots, z_d] \in \mathbb{R}^d $.}
\end{array}
\end{equation*}
Besides, we let:
\begin{align*}
    & \mathbb{B}^d := \left\{ \begin{array}{l|l}
        y \in \R^d & |y| < 1
    \end{array} \right\} ~ \mbox{ and } ~ \mathbb{S}^{d -1} := \left\{ \begin{array}{l|l}
        y \in \R^d & |y| = 1
    \end{array} \right\}.
\end{align*}
We denote by $\mathcal{L}^{d}$ the $ d $-dimensional Lebesgue measure, and we denote by $ \mathcal{H}^{d} $ the $ d $-dimensional Hausdorff measure.  In particular, the measure theoretical phrases, such as ``a.e.'', ``$dt$'', and ``$dx$'', and so on, are all with respect to the Lebesgue measure in each corresponding dimension. Also on a Lipschitz-surface $ S $, the phrase ``a.e.'' is with respect to the Hausdorff measure in each corresponding Hausdorff dimension. In particular, if $S$ is $C^1$-surface, then we simply denote by $dS$ the area-element of the integration on $S$.

For a Borel set $ E \subset \R^d $, we denote by $ \chi_E : \R^d \longrightarrow \{0, 1\} $ the characteristic function of $ E $. Additionally, for a distribution $ \zeta $ on an open set in $ \R^d $ and any $i \in \{ 1,\dots,d \}$, let $ \partial_i \zeta$ be the distributional differential with respect to $i$-th variable of $\zeta$. As well as we consider, the differential operators, such as $\nabla,\ \diver, \ \nabla^2$, and so on, in distributional sense.
\bigskip

\noindent
\underline{\textbf{\textit{Abstract notations. (cf. \cite[Chapter II]{MR0348562})}}}
For an abstract Banach space $ X $, we denote by $ |\cdot|_{X} $ the norm of $ X $, and denote by $ \langle \cdot, \cdot \rangle_X $ the duality pairing between $ X $ and its dual $ X^* $. In particular, when $ X $ is a Hilbert space, we denote by $ (\cdot,\cdot)_{X} $ the inner product of $ X $. 

For two Banach spaces $ X $ and $ Y $,  let $  \mathscr{L}(X; Y)$ be the Banach space of bounded linear operators from $ X $ into $ Y $. 

For Banach spaces $ X_1, \dots, X_d $ with $ 1 < d \in \N $, let $ X_1 \times \dots \times X_d $ be the product Banach space endowed with the norm $ |\cdot|_{X_1 \times \cdots \times X_d} := |\cdot|_{X_1} + \cdots +|\cdot|_{X_d} $. However, when all $ X_1, \dots, X_d $ are Hilbert spaces, $ X_1 \times \dots \times X_d $ denotes the product Hilbert space endowed with the inner product $ (\cdot, \cdot)_{X_1 \times \cdots \times X_d} := (\cdot, \cdot)_{X_1} + \cdots +(\cdot, \cdot)_{X_d} $ and the norm $ |\cdot|_{X_1 \times \cdots \times X_d} := \bigl( |\cdot|_{X_1}^2 + \cdots +|\cdot|_{X_d}^2 \bigr)^{\frac{1}{2}} $. In particular, when all $ X_1, \dots,  X_d $ coincide with a Banach space $ Y $, the product space $X_1 \times \dots \times X_d$ is simply denoted by $[Y]^d$.
\bigskip

\noindent
\underline{\textbf{\textit{Basic notations.}}}
Let $0 < T < \infty$ be a fixed constant of time, and let $N \in \{ 1,2,3 \}$ is a fixed dimension. Let $\Omega \subset \R^N$ be a bounded domain with a boundary $\Gamma := \partial\Omega$, and when $N > 1$, $\Gamma$ has $C^\infty$-regularity with the unit outer normal $n_\Gamma$. Additionally, as notations of base spaces, we let:
\begin{equation*}
    H := L^2(\Omega), \ V := H^1(\Omega), ~ 
    \sH := L^2(0,T;H),  \mbox{ and } \sV := L^2(0,T;V).
\end{equation*}
Let $W_0 \subset H^2(\Omega)$ be the closed linear subspace of $H$, given by:
\begin{equation*}
  W_0 := \{ z \in H^2(\Omega)\,|\, \nabla z \cdot n_\Gamma = 0 \mbox{ on } \Gamma \}.
\end{equation*}
Let $A_N$ be a differential operator, defined as:
\begin{equation*}
    A_N: \, z \in W_0 \subset H \mapsto A_N z := - \mathit{\Delta} z \in H.
\end{equation*}
It is well-known that $ A_N \subset H \times H $ is linear, positive, and self-adjoint, and the domain $ W_0 $ is a Hilbert space, endowed with the inner product:
    \begin{gather*}
        (z_1, z_2)_{W_0} := (z_1, z_2)_H +(A_Nz_1, z_2)_H ~ \bigl( = (z_1, z_2)_V \bigr), \mbox{ for $ z_k \in W_0 $, $k = 1, 2$.}
    \end{gather*}
    Moreover, there exists a positive constant $ C_0 $ such that:
    \begin{gather}\label{embb01}
        |z|_{H^2(\Omega)}^2 \leq C_0 \bigl( |A_N z|_H^2 + |z|_H^2 \bigr),~ \mbox{ for all $ z \in W_0 $.}
    \end{gather}
\medskip

\noindent
\underline{\textbf{\textit{Notations for the time-discretization.}}}
Let $\tau > 0$ be a constant of the time step-size, and let $\{ t_i \}_{i=0}^\infty \subset [0,\infty)$ be the time sequence defined as:
\begin{equation*}
  t_i := i\tau,\ i=0,1,2,\ldots.
\end{equation*}
Let $X$ be a Banach space. Then, for any sequence $\{ [t_i,z_i] \}_{i=0}^\infty \subset[0,\infty) \times X$, we define the \textit{forward time-interpolation} $[\overline{z}]_\tau \in L^\infty_\mathrm{loc}([0,\infty); X)$, the \textit{backward time-interpolation} $[\underline{z}]_\tau \in L^\infty_\mathrm{loc}([0,\infty);X)$ and the \textit{linear time-interpolation} $[z]_\tau \in W^{1,2}_\mathrm{loc}([0,\infty);X)$, by letting:
\begin{equation*}
  \left\{\begin{aligned}
    &[\overline{z}]_\tau(t) := \chi_{(-\infty,0]} z_0 + \sum_{i=1}^\infty \chi_{(t_{i-1}, t_i]}(t) z_i, \\
    &[\underline{z}]_\tau(t) := \sum_{i=0}^\infty \chi_{(t_{i}, t_{i+1}]}(t) z_{i}, \\
    &[z]_\tau (t) := \sum_{i=1}^\infty \chi_{[t_{i-1},t_i)}(t) \left(\frac{t-t_{i-1}}{\tau} z_{i} + \frac{t_i - t}{\tau} z_{i-1}\right),
  \end{aligned}\right. ~{\rm in}~ X,\ {\rm for}~ t \geq 0, \label{eq:tI}
\end{equation*}
respectively.

In the meantime, for any $ q \in [1, \infty) $ and any $ \zeta \in L_\mathrm{loc}^q([0, \infty); X) $, we denote by $ \{ \zeta_i \}_{i = 0}^\infty \subset X $ the sequence of time-discretization data of $ \zeta $, defined as:
\begin{subequations}\label{tI}
\begin{align}\label{tI01}
    & \zeta_0 := 0 \mbox{ in $X$, and }\zeta_i := \frac{1}{\tau} \int_{t_{i -1}}^{t_i} \zeta(\varsigma) \, d \varsigma ~ \mbox{ in $ X $, ~ for $ i = 1, 2, 3, \dots $.}
\end{align}
As is easily checked, the time-interpolations $ [\overline{\zeta}]_\tau, [\underline{\zeta}]_\tau \in L^q_\mathrm{loc}([0, \infty); X) $ for the above $ \{ \zeta_i \}_{i = 0}^\infty $ fulfill that:
\begin{align}\label{tI02}
    & [\overline{\zeta}]_\tau \to \zeta \mbox{ and } [\underline{\zeta}]_\tau \to \zeta \mbox{ in $ L^q_\mathrm{loc}([0, \infty); X) $, as $ \tau \downarrow 0 $.}
\end{align}
\end{subequations}

\section{Main results}
\label{sec:main}

In this paper, the main assertions are discussed under the following assumptions.
\vspace{2ex}

\begin{itemize}
%  \item[(A0)] $0 < T < \infty$ is a fixed constant of time, and $N \in \{ 1,2,3 \}$ is a fixed dimension. $\Omega \subset \R^N$ is a bounded domain with a boundary $\Gamma := \partial\Omega$, and when $N > 1$, $\Gamma$ has $C^\infty$-regularity with the unit outer normal $n_\Gamma$. 
  \item[(A1)] $\mu > 0$ and $\nu > 0$ are fixed constants.
    \vspace{1ex}
  \item[(A2)] $g: \R \longrightarrow \R$ is a locally Lipschitz continuous function with a nonnegative primitive $G \in C^1(\R)$. Moreover, $g$ satisfies the following condition:
  \begin{equation*}
    \liminf_{\xi \downarrow -\infty} g(\xi) = -\infty, \ \limsup_{\xi \uparrow \infty} g(\xi) = \infty.
  \end{equation*}
  \item[(A3)] $\alpha_0: \R \longrightarrow (0,\infty)$ is a locally Lipschitz continuous function, and $\alpha : \R \longrightarrow [0,\infty)$ is a $C^1$-class convex function, such that:
  \begin{equation*}
    \alpha'(0) = 0 \mbox{ and } \delta_{\alpha} := \inf \alpha_0(\R) > 0.
  \end{equation*}
  \item[(A4)] $u, v \in L_{\rm loc}^2([0,\infty);H)$, and $u \in L^\infty(Q)$.
    \vspace{1ex}
  \item[(A5)] The initial data $[\eta_0, \theta_0]$ belong to the class $[W_0]^2 \subset [H^2(\Omega)]^2$.
\end{itemize}
\vspace{2ex}

Now, the main results are stated as follows.
\medskip

\begin{mainThm}[Existence and regularity]\label{mainThm1}
    Under the assumptions (A1)--(A5), the system (S) admits a solution $[\eta, \theta] \in [\sH]^2 $ in the following sense.
    \vspace{2ex}
  \begin{itemize}
    \item[(S0)] $[\eta, \theta] \in \bigl[W^{1,2}(0,T;W_0) \cap L^\infty(Q)\bigr] \times \bigl[ W^{1,2}(0,T; V) \times L^\infty(0,T; W_0) \bigr]$.
    \vspace{1ex}
    \item[(S1)] $\eta$ solves the following variational identity:
    \begin{gather*}
      (\partial_t \eta(t), \varphi)_H +  (\nabla (\eta + \mu^2 \partial_t \eta)(t), \nabla \varphi)_{[H]^N} + (g(\eta(t)), \varphi)_H
        \\
        + (\alpha'(\eta(t))|\nabla \theta(t)|, \varphi)_H  = (u(t), \varphi)_H,
        \\
        \mbox{ for any } \varphi \in V, \mbox{ and a.e. } t \in (0,T).
    \end{gather*} 
    \item[(S2)] $\theta$ solves the following variational inequality:
    \begin{gather*}
        \hspace{-20ex}\bigl((\alpha_0(\eta) \partial_t \theta)(t), \theta(t) - \psi\bigr)_H + \nu^2 (\nabla \partial_t \theta(t), \nabla (\theta(t) - \psi))_{[H]^N}
      \\
       + \int_\Omega \alpha(\eta(t))|\nabla \theta(t)|\,dx \leq \int_\Omega \alpha(\eta(t)) |\nabla \psi|\,dx + (v(t), \theta(t) - \psi)_H, 
        \\ 
        \mbox{ for any } \psi \in V, \mbox{ and a.e. }t \in (0,T).
    \end{gather*}
      \vspace{-3ex}
    \item[(S3)] $[\eta(0),\theta(0)] = [\eta_0, \theta_0]$ in $[H]^2$.
  \end{itemize}
\end{mainThm}
\vspace{1ex}

\begin{mainThm}[Uniqueness and continuous dependence]
    \label{mainThm2}
      Under the assumptions (A1)--(A5), let $[\eta^k,\theta^k], \, k = 1,2$ be two solutions to (S) with two initial values $[\eta_0^k, \theta_0^k]$ and two forcings $[u^k, v^k]$, $k = 1,2$. Then, there exists a constant $C_1 = C_1(\nu) > 0$, depending on $\nu$, such that:
  \begin{gather*}
    J(t) \leq C_1(\nu) \bigl( J(0) + |u^1 - u^2|_\sH^2 + |v^1 - v^2|_\sH^2 \bigr), 
      \\ 
      \mbox{ for any } t \in [0,T] \mbox{ and any } T > 0,
  \end{gather*}
  where
  \begin{align*}
    J(t) := &|(\eta^1 - \eta^2)(t) |_H^2 + \mu^2|\nabla (\eta^1 - \eta^2)(t) |_{[H]^N}^2 + |\sqrt{\alpha_0(\eta^1)}(\theta^1 - \theta^2)(t)|_H^2
    \\
    &\quad + \nu^2| \nabla (\theta^1 - \theta^2)(t) |_{[H]^N}^2, \ \mbox{ for any } t \geq 0.
  \end{align*}
\end{mainThm}

\section{Approximating method}
\label{sec:approx}

In the Main Theorems, the solution to (S) will be obtained by means of the time-discretization method. In this light, let $\tau \in (0,1)$ be a constant of the time-step size, and let $\varepsilon \in (0,1)$ be a relaxation constant. Based on this, we adopt the following time-discretization scheme (AP)$_\tau^\varepsilon$, as our approximating problem of (S).
\vspace{2ex}

\noindent
(AP)$_\tau^\varepsilon$: To find $\{ [\eta_i, \theta_i] \}_{ i=0 }^\infty \subset [W_0]^2$ satisfying:
\begin{gather}
    \label{AP_eta}
    \frac{\eta_i^\varepsilon - \eta_{i-1}^\varepsilon}{\tau} +A_N \left( \eta_i^\varepsilon + \frac{\mu^2}{\tau} (\eta_i^\varepsilon - \eta_{i-1}^\varepsilon) \right) + g(\cT_M\eta_i^\varepsilon) 
    \\
    \nonumber
    + \alpha'(\cT_M\eta_i^\varepsilon) \gamma_\varepsilon(\nabla \theta_i^\varepsilon) = u_i \mbox{ in } H,
    %\\
    %&\nabla \eta_i^\varepsilon \cdot n_\Gamma = 0 \mbox{ on } \Gamma,
    \\[1ex]
    \frac{\alpha_0(\cT_M\eta_{i-1}^\varepsilon)}{\tau}(\theta_i^\varepsilon - \theta_{i-1}^\varepsilon) - \diver \bigl( \widetilde \alpha_M(\eta) \nabla \gamma_\varepsilon(\nabla \theta_i^\varepsilon) \bigr) 
  \label{AP_theta}
    \\
    + \frac{\nu^2}{\tau}A_N (\theta_i^\varepsilon - \theta_{i-1}^\varepsilon) = v_i \mbox{ in }H,
    %\\
    %&\nabla \theta_i^\varepsilon \cdot n_\Gamma = 0 \mbox{ on } \Gamma,
    \nonumber
    \\[1ex]
  \mbox{for } i = 1,2,3, \dots, \mbox{ subject to } [\eta_0^\varepsilon, \theta_0^\varepsilon] = [\eta_0, \theta_0] \mbox{ in } [H]^2.
    \nonumber
\end{gather}
In this context, $\gamma_\varepsilon \in C^{0,1}(\R^N) \cap C^\infty(\R^N)$ is a smooth approximation of the Euclidean norm $|\cdot| \in C^{0,1}(\R^N)$, defined as:
\begin{equation*}
  \gamma_\varepsilon: y \in \R^N \mapsto \gamma_\varepsilon(y) := \sqrt{\varepsilon^2 + |y|^2} \in [0,\infty).
\end{equation*}
Also, we define an approximating free-energy $\F_\varepsilon$ on $[H]^2$, by setting:
\begin{gather}
  \label{free_eng} 
    \F_\varepsilon:[\eta,\theta] \in [H]^2 \mapsto \F_\varepsilon(\eta,\theta) := \frac{1}{2} \int_\Omega |\nabla \eta|^2\,dx + \int_\Omega \widetilde G_M(\eta) 
    \\
    + \int_\Omega \widetilde \alpha_M(\eta) \gamma_\varepsilon(D\theta), \ \mbox{ for any } 0 < \varepsilon < 1.
    \nonumber
\end{gather}
where $\widetilde \alpha_M \in C^{1,1}(\R)$ and $\widetilde G_M \in C^{1,1}(\R)$ are nonnegative primitives of $\alpha \circ \cT_M \in W^{1,\infty}(\R)$ and $G \circ \cT_M \in W^{1,\infty}(\R)$, respectively. Finally, $u_i, v_i$ are given as in \eqref{tI}.

\medskip
The solution to (AP)$_\tau^\varepsilon$ is given as follows.

\begin{definition}
  The sequence of functions $\{ [\eta_i^\varepsilon, \theta_i^\varepsilon] \}_{i=0}^\infty$ is called a solution to (AP)$_\tau^\varepsilon$ iff. $\{ [\eta_i^\varepsilon, \theta_i^\varepsilon] \}_{i=0}^\infty \subset [W_0]^2$, and $[\eta_i^\varepsilon, \theta_i^\varepsilon]$ fulfills \eqref{AP_eta} and \eqref{AP_theta} for any $i = 1,2,3,\dots$.
\end{definition}

In this paper, the following theorem will plays an important role for the proof of Main Theorems.

\begin{theorem}[Solvability of the approximating problem] \label{AP}
  There exists a sufficiently small constant $\tau_0 \in (0,1)$ such that for any $\tau \in (0,\tau_0)$ and $\varepsilon \in (0,1)$, (AP)$_\tau^\varepsilon$ admits a unique solution $\{ [\eta_i^\varepsilon, \theta_i^\varepsilon] \}_{i = 0}^\infty$. Additionally, the following energy inequality holds:
  \begin{gather}
    \frac{1}{4\tau} |\eta_i^\varepsilon - \eta_{i-1}^\varepsilon|_H^2 + \frac{\mu^2}{\tau} |\nabla (\eta_i^\varepsilon - \eta_{i-1}^\varepsilon) |_{[H]^N}^2 + \frac{\delta_{\alpha}}{2\tau} |\theta_i^\varepsilon - \theta_{i-1}^\varepsilon|_H^2 
      \label{Energy1}
    \\
    + \frac{\nu^2}{\tau} |\nabla (\theta_i^\varepsilon - \theta_{i-1}^\varepsilon)|_{[H]^N}^2 + \F_\varepsilon(\eta_i^\varepsilon, \theta_i^\varepsilon) 
      \nonumber
    \\
    \leq \F_\varepsilon(\eta_{i-1}^\varepsilon, \theta_{i-1}^\varepsilon) + \frac{\tau}{2}|u_i|_H^2 + \frac{\tau}{2\delta_{\alpha}}|v_i|_H^2, \ \mbox{ for any } i =  1,2,3,\dots.
      \nonumber
  \end{gather}
\end{theorem}

\Cref{AP} is proved through several lemmas.

\begin{lemma}\label{Lem_AP1}
  For arbitrary $\widetilde \theta \in V$, $\widetilde \eta_0 \in W_0$ and $\widetilde u \in H$, we consider the following elliptic problem:
  \begin{equation}
    \frac{1}{\tau} (\eta - \widetilde \eta_0) + A_N \left( \eta + \frac{\mu^2}{\tau}(\eta - \widetilde \eta_0) \right) + g(\cT_M \eta) + \alpha'(\cT_M \eta) \gamma_\varepsilon(\nabla \widetilde \theta) = \widetilde u, \mbox{ in } H. \label{AP1}
  \end{equation}
  Then, there exists a small constant $\tau_1 \in (0,1)$, depending only on $|g'|_{L^\infty(-M,M)}$, and for any $0 < \tau < \tau_1$, the elliptic problem \eqref{AP1} admits a unique solution $\eta \in W_0$.
\end{lemma}

\begin{proof}
  First, for any $\eta^\dagger \in V$, we define a functional $\Upsilon: H \longrightarrow (-\infty, \infty]$ as follows:
  \begin{equation*}
    \Upsilon: z \in H \mapsto \Upsilon(z) := \left\{ \begin{aligned}
      &\frac{1}{2\tau} \int_\Omega |z|^2 \,dx + \frac{1}{2} \int_\Omega |\nabla z|^2 \,dx + \frac{\mu^2}{2\tau} \int_\Omega |\nabla (z - \widetilde \eta_0)|^2 \,dx
      \\
      &\quad + \int_\Omega \widetilde g(\cT_M \eta^\dagger) z \,dx + \int_\Omega \widetilde \alpha_M(z) \gamma_\varepsilon(\nabla \widetilde \theta) \,dx
      \\
      &\quad - \int_\Omega \widetilde u z \,dx, \ \mbox{ if } z \in V,
      \\
      &\infty, \mbox{ otherwise}.
    \end{aligned} \right.
  \end{equation*}
  As is easily checked, $\Upsilon$ is proper, l.s.c., strictly convex, and coercive, and its unique minimizer solves the following elliptic equation:
  \begin{equation}
    \frac{1}{\tau} (\eta - \widetilde \eta_0) + A_N \left( \eta + \frac{\mu^2}{\tau}(\eta - \widetilde \eta_0) \right) + g(\cT_M \eta^\dagger) + \alpha'(\cT_M \eta) \gamma_\varepsilon(\nabla \widetilde \theta) = \widetilde u, \mbox{ in } H. \label{AP1.1}
  \end{equation}

  Now, we define an operator $S_\tau:V \longrightarrow H^2(\Omega)$ which maps any $\eta^\dagger \in V$ to the unique solution to \eqref{AP1.1}, and consider the smallness condition of $\tau$ for $S$ to be contractive. Here, let $\eta_k := S_\tau \eta_k^\dagger \in H^2(\Omega)$, $k=1,2$. By taking differences of \eqref{AP1.1}, multiplying both sides by $\eta_1 - \eta_2$ and applying Young's inequality, we see from (A1) and (A2) that:
  \begin{gather*}
    \frac{1 \wedge \mu^2}{2\tau}|\eta_1 - \eta_2|_V^2 \leq \frac{\tau|g'|_{L^\infty(-M,M)}^2}{2}|\eta_1^\dagger - \eta_2^\dagger|_H^2.
  \end{gather*}
  Therefore, if we assume that
  \begin{equation}
    0 < \tau < \tau_1 := \left( \frac{1 \wedge \mu^2}{|g|_{L^\infty(-M,M)}^2} \right)^{\frac{1}{2}}, \label{FP1}
  \end{equation}
  then the mapping $S_\tau$ becomes a contraction mapping from $V$ into itself. Therefore, applying Banach's fixed point theorem, we find a unique fixed point $\widetilde \eta \in V$ of $S_\tau$ under the condition \eqref{FP1}. The identity $S_\tau \widetilde \eta = \widetilde \eta$ implies that $\widetilde \eta$ is the unique solution to \eqref{AP1}. 
\end{proof}

\begin{lemma}\label{Lem_AP2}
  For arbitrary $\widetilde \eta \in H^2(\Omega)$, $\widetilde \theta_0 \in W_0$ and $\widetilde v \in H$, we consider the following elliptic equation:
  \begin{equation}
    \left\{ \begin{aligned}
      &\alpha_0 (\cT_M \widetilde \eta) \frac{\theta - \widetilde \theta_0}{\tau} - \diver \left( \widetilde \alpha_M (\widetilde \eta) \nabla \gamma_\varepsilon(\nabla \theta) + \frac{\nu^2}{\tau}\nabla(\theta - \widetilde \theta_0) \right) = \widetilde v \mbox{ a.e. in } \Omega,
      \\
      &\nabla \theta \cdot n_\Gamma = 0 \mbox{ a.e. on } \Gamma.
    \end{aligned} \right. \label{AP2}
  \end{equation}
  Then, for any $0 < \tau < 1$, \eqref{AP2} admits a unique solution $\theta \in W_0$.
\end{lemma}

\begin{proof}
  Let us consider a proper, l.s.c., strictly convex, and coercive function $\Upsilon_*$ defined as follows:
  \begin{equation*}
    \Upsilon_* : z \in H \mapsto \Upsilon_*(z) := \left\{ \begin{aligned}
      &\frac{1}{2\tau} \int_\Omega \alpha_0(\cT_M \widetilde \eta) |z - \widetilde \theta_0|^2\,dx + \int_\Omega \widetilde \alpha_M(\widetilde \eta) \gamma_\varepsilon(\nabla z) \,dx
      \\
      &\quad + \frac{\nu^2}{2\tau} \int_\Omega |\nabla (z - \widetilde \theta_0)|^2\,dx - \int_\Omega \widetilde v z\,dx, \ \mbox{ if } z \in V,
      \\
      &\infty, \mbox{ otherwise. }
    \end{aligned} \right.
  \end{equation*}
  As is discussed in \cite[Theorem 1]{aiki2023class}, the unique minimizer $\widetilde \theta$ of $\Upsilon_*$ solves \eqref{AP2}, and $\widetilde \theta$ belongs to $W_0$.
\end{proof}

\begin{proof}[Proof of \cref{AP}]
  Let us fix any $\tau \in (0,\tau_1)$ and any $\varepsilon \in (0,1)$. Then, for any $i \in \N$, we can obtain $\theta_i^\varepsilon \in W_0$ by applying \cref{Lem_AP2} in the case that:
  \begin{equation*}
    \widetilde \eta = \eta_{i-1}^\varepsilon, \ \widetilde \theta_0 = \theta_{i-1}^\varepsilon \mbox{ and } \widetilde v = v_i \mbox{ in } H.
  \end{equation*}
  Moreover, for any $i\in \N$, the component $\eta_i^\varepsilon \in W_0$ can be obtained by applying \cref{Lem_AP1} in the case that:
  \begin{equation*}
    \widetilde \theta = \theta_i^\varepsilon, \ \widetilde \eta_0 := \eta_{i-1}^\varepsilon \mbox{ and } \widetilde u = u_i \mbox{ in } H.
  \end{equation*}
  Thus we can find the unique solution $\{ [\eta_i^\varepsilon, \theta_i^\varepsilon] \}_{i=0}^\infty \subset [W_0]^2$ to (AP)$_\tau^\varepsilon$.

  Next, we verify the inequality \eqref{Energy1}. Multiplying both sides of \eqref{AP_eta} with $\eta_i^\varepsilon - \eta_{i-1}^\varepsilon$, we see that:
\begin{gather}
    \frac{1}{2\tau}|\eta_i^\varepsilon - \eta_{i-1}^\varepsilon|_H^2 + \frac{\mu^2}{\tau} |\nabla (\eta_i^\varepsilon - \eta_{i-1}^\varepsilon)|_{[H]^N}^2 + \frac{1}{2}|\nabla \eta_i^\varepsilon|_{[H]^N}^2 - \frac{1}{2}|\nabla \eta_{i-1}^\varepsilon|_{[H]^N}^2 
      \label{Energy2}
    \\
    + \bigl( g(\cT_M \eta_i^\varepsilon), \eta_i^\varepsilon - \eta_{i-1}^\varepsilon \bigr)_H + \bigl( \alpha'(\cT_M \eta_i^\varepsilon) \gamma_\varepsilon(\nabla \theta_i^\varepsilon), \eta_i^\varepsilon - \eta_{i-1}^\varepsilon \bigr)_H \leq \frac{\tau}{2}|u_i|_H^2, 
    \nonumber
    \\
    \mbox{for } i = 1,2,3,\dots.
    \nonumber
  \end{gather}
  via the following computations:
  \begin{gather*}
%    \frac{1}{\tau}(\eta_i^\varepsilon - \eta_{i-1}^\varepsilon, \eta_i^\varepsilon - \eta_{i-1}^\varepsilon)_H = \frac{1}{\tau}|\eta_i^\varepsilon - \eta_{i-1}^\varepsilon|_H^2,
%    \\
    (\nabla \eta_i^\varepsilon, \nabla (\eta_i^\varepsilon - \eta_{i-1}^\varepsilon))_{[H]^N} \geq \frac{1}{2}\bigl( |\nabla \eta_i^\varepsilon|_{[H]^N}^2 - |\nabla \eta_{i-1}^\varepsilon|_{[H]^N}^2 \bigr),
    \\
    \frac{\mu^2}{\tau}(\nabla(\eta_i^\varepsilon - \eta_{i-1}^\varepsilon), \nabla(\eta_i^\varepsilon - \eta_{i-1}^\varepsilon))_{[H]^N} = \frac{\mu^2}{\tau}|\nabla (\eta_i^\varepsilon - \eta_{i-1}^\varepsilon)|_{[H]^N}^2,
  \end{gather*}
  and
  \begin{gather*}
    (u_i, \eta_i^\varepsilon - \eta_{i-1}^\varepsilon)_H \leq \frac{1}{2\tau}|\eta_i^\varepsilon - \eta_{i-1}^\varepsilon|_H^2 + \frac{\tau}{2}|u_i|_H^2.
  \end{gather*}
  In addition, by using (A1), it is obtained that:
  \begin{gather}
    \label{Energy3}
      \bigl( g(\cT_M \eta_i^\varepsilon), \eta_i^\varepsilon - \eta_{i-1}^\varepsilon \bigr)_H \geq \int_\Omega \widetilde G_M(\eta_i^\varepsilon)\,dx - \int_\Omega \widetilde G_M(\eta_{i-1}^\varepsilon)\,dx 
    \\
      \nonumber
    + \bigl( g(\cT_M \eta_i^\varepsilon) - g(\cT_M \eta_{i-1}^\varepsilon), \eta_i^\varepsilon - \eta_{i-1}^\varepsilon \bigr)_H - \frac{1}{2}|g'|_{L^\infty(-M,M)}|\eta_i^\varepsilon - \eta_{i-1}^\varepsilon|_H^2
    \\
      \nonumber
    \geq \int_\Omega \widetilde G_M(\eta_i^\varepsilon)\,dx - \int_\Omega \widetilde G_M(\eta_{i-1}^\varepsilon)\,dx - \frac{3}{2}|g'|_{L^\infty(-M,M)}|\eta_i^\varepsilon - \eta_{i-1}^\varepsilon|_H^2,
    \\
      \nonumber
    \mbox{ for } i = 1,2,3,\dots,
  \end{gather}%\noeqref{Energy3}
  and by the convexity of $\widetilde \alpha_M$,
  \begin{gather}
    \bigl( \alpha'(\cT_M \eta_i^\varepsilon) \gamma_\varepsilon(\nabla \theta_i^\varepsilon), \eta_i^\varepsilon - \eta_{i-1}^\varepsilon \bigr)_H \geq \int_\Omega \widetilde \alpha_M (\eta_i^\varepsilon) \gamma_\varepsilon(\nabla \theta_i^\varepsilon) \,dx
      \label{Energy4}
    \\
     - \int_\Omega \widetilde \alpha_M(\eta_{i-1}^\varepsilon) \gamma_\varepsilon(\nabla \theta_i^\varepsilon)\,dx,
    \mbox{ for } i = 1,2,3,\dots.
      \nonumber
  \end{gather}
  On account of \eqref{Energy2}--\eqref{Energy4}, it is inferred that:
  \begin{gather}
    \left( \frac{1}{2} - \frac{3\tau}{2} |g'|_{L^\infty(-M,M)} \right)\frac{1}{\tau}|\eta_i^\varepsilon - \eta_{i-1}^\varepsilon|_H^2 + \frac{\mu^2}{\tau} |\nabla (\eta_i^\varepsilon - \eta_{i-1}^\varepsilon)|_{[H]^N}^2 
      \label{Energy5}
    \\
    + \frac{1}{2}|\nabla \eta_i^\varepsilon|_{[H]^N}^2 + \int_\Omega \widetilde G_M(\eta_i^\varepsilon) \,dx + \int_\Omega \widetilde \alpha_M(\eta_i^\varepsilon) \gamma_\varepsilon(\nabla \theta_i^\varepsilon)\,dx 
      \nonumber
    \\
    \leq \frac{1}{2}|\nabla\eta_{i-1}^\varepsilon|_{[H]^N}^2 + \int_\Omega \widetilde G_M(\eta_{i-1}^\varepsilon)\,dx + \int_\Omega \widetilde \alpha_M(\eta_{i-1}^\varepsilon) \gamma_\varepsilon(\nabla \theta_i^\varepsilon)\,dx + \frac{\tau}{2}|u_i|_H^2, 
      \nonumber
    \\
    \mbox{ for } i = 1,2,3,\dots.
      \nonumber
  \end{gather}

  On the other hand, by multiplying both sides of \eqref{AP_theta} by $\theta_i^\varepsilon - \theta_{i-1}^\varepsilon$, and using (A3) and the convexity of $\gamma_\varepsilon$, we have
  \begin{align}
    &\frac{\delta_{\alpha}}{2\tau} |\theta_i^\varepsilon - \theta_{i-1}^\varepsilon|_H^2 + \frac{\nu^2}{\tau} |\nabla(\theta_i^\varepsilon - \theta_{i-1}^\varepsilon)|_{[H]^N}^2 + \int_\Omega \widetilde \alpha_M(\eta_{i-1}^\varepsilon) \gamma_\varepsilon(\nabla \theta_i^\varepsilon)\,dx \label{Energy6}
    \\
    & \leq \int_\Omega \widetilde \alpha_M(\eta_{i-1}^\varepsilon) \gamma_\varepsilon(\nabla \theta_{i-1}^\varepsilon)\,dx + \frac{\tau}{2\delta_{\alpha}}|v_i|_H^2, \ \mbox{ for } i = 1,2,3,\dots,
      \nonumber
  \end{align}
  via the following computation:
  \begin{gather*}
    \frac{1}{\tau}(\alpha_0(\cT_M\eta_{i-1}) (\theta_i^\varepsilon - \theta_{i-1}^\varepsilon), \theta_i^\varepsilon - \theta_{i-1}^\varepsilon)_H \geq \frac{\delta_{\alpha}}{\tau} |\theta_i^\varepsilon - \theta_{i-1}^\varepsilon|_H^2,
%    \\
%    \frac{\nu^2}{\tau}(\nabla (\theta_i^\varepsilon - \theta_{i-1}^\varepsilon), \nabla (\theta_i^\varepsilon - \theta_{i-1}^\varepsilon))_{[H]^N} = \frac{\nu^2}{\tau} |\nabla (\theta_i^\varepsilon - \theta_{i-1}^\varepsilon)|_{[H]^N}^2,
  \end{gather*}
  \begin{gather*}
    (\widetilde \alpha_M(\eta_{i-1}^\varepsilon) \nabla \gamma_\varepsilon(\nabla\theta_i^\varepsilon), \nabla( \theta_i^\varepsilon - \theta_{i-1}^\varepsilon))_{[H]^N} \qquad\qquad
    \\
    \qquad\qquad\qquad\geq \int_\Omega \widetilde \alpha_M(\eta_{i-1}^\varepsilon) \gamma_\varepsilon(\nabla \theta_i^\varepsilon)\,dx - \int_\Omega \widetilde \alpha_M(\eta_{i-1}^\varepsilon) \gamma_\varepsilon(\nabla \theta_{i-1}^\varepsilon) \,dx,
  \end{gather*}
  and
  \begin{gather*}
    (u_i, \theta_i^\varepsilon - \theta_{i-1}^\varepsilon)_H \leq \frac{\delta_{\alpha}}{2\tau}|\theta_i^\varepsilon - \theta_{i-1}^\varepsilon|_H^2 + \frac{\tau}{2\delta_{\alpha}}|u_i|_H^2.
  \end{gather*}
  
  Now, let us set $\tau_0$ as 
  \begin{gather*}
    \tau_0 := \min\left\{ \tau_1, \frac{1}{6|g'|_{L^\infty(-M,M)}} \right\}, 
     \mbox{ with the constant $\tau_1$ as in \cref{Lem_AP1}.}
  \end{gather*}
  Then, from \eqref{Energy5} and \eqref{Energy6}, we obtain that:
  \begin{align*}
    &\frac{1}{4\tau} |\eta_i^\varepsilon - \eta_{i-1}^\varepsilon|_H^2 + \frac{\mu^2}{\tau} |\nabla (\eta_i^\varepsilon - \eta_{i-1}^\varepsilon) |_{[H]^N}^2 + \frac{\delta_{\alpha}}{2\tau} |\theta_i^\varepsilon - \theta_{i-1}^\varepsilon|_H^2
    \\
    &\quad + \frac{\nu^2}{\tau} |\nabla (\theta_i^\varepsilon - \theta_{i-1}^\varepsilon)|_{[H]^N}^2 + \F_\varepsilon(\eta_i^\varepsilon, \theta_i^\varepsilon) 
    \\
    &\quad \leq \F_\varepsilon(\eta_{i-1}^\varepsilon, \theta_{i-1}^\varepsilon) + \frac{\tau}{2}|u_i|_H^2 + \frac{\tau}{2\delta_{\alpha}}|v_i|_H^2, \ \mbox{ for } i =  1,2,3,\dots.
  \end{align*}
  
  Thus, we conclude \cref{AP}.  
\end{proof}

\section{Proofs of Main Theorems}

In this section, we will provide proofs of Main Theorems. We set 
\begin{equation}\label{n_tau}
    n_\tau := \min \{ n \in \N \,|\, n \tau \geq T \}.
\end{equation}
Additionally, under the notations as in \cref{AP}, we invoke \eqref{tI01} and \eqref{tI02}, and take a small constant $\tau_* \in (0,\tau_0)$, such that:
\begin{equation*}
  \tau \sum_{i=1}^{n_\tau} \bigl( |u_i|_H^2 + |v_i|_H^2 \bigr) \leq |u|_\sH^2 + |v|_\sH^2 + 1, \mbox{ whenever } \tau \in (0,\tau_*).
\end{equation*}

\subsection{Proof of Main Theorem 1}

Before we deal with the proof, we will prepare some lemmas. Hereafter, based on (A1), (A3), (A4) and (A5), we set the constant $M > 0$ of truncation, so large to satisfy that:
\begin{equation}
  M \geq |\eta_0|_{L^\infty(\Omega)}, \ g(M) \geq |u|_{L^\infty(Q)}, \mbox{ and } g(-M) \leq -|u|_{L^\infty(Q)}. \label{ken02}
\end{equation}
Then, it immediately follows that:
\begin{equation}
  \widetilde G_M (\eta_0) = G(\eta_0), \mbox{ and } \widetilde \alpha_M (\eta_0) = \alpha(\eta_0). \label{dm02}
\end{equation}

\begin{lemma}\label{Lem_eta}
  Let $\tau \in (0,\tau_*)$. Then, there exists a constant $C_2 > 0$, independent of $\varepsilon$ and $ \tau$, such that:
  \begin{gather}
    \frac{1}{\tau}\sum_{i=1}^{n_\tau} |\eta_i^\varepsilon - \eta_{i-1}^\varepsilon|_{H^2(\Omega)}^2 \leq C_2 \bigl( |\eta_0|_{H^2(\Omega)}^2 + |\theta_0|_V^2 + |u|_\sH^2 + |v|_\sH^2 + 1 \bigr). 
      \label{H2_eta1}
  \end{gather}
\end{lemma}

\begin{proof}
  First, from the definition of $\F_\varepsilon$ \eqref{free_eng}, \eqref{dm02}, and the embedding $W_0 \subset L^\infty(\Omega)$ under $ N \leq 3 $, it is seen that:
  \begin{align*}
    \F_\varepsilon(\eta_0, \theta_0) &= \frac{1}{2} \int_\Omega |\nabla \eta_0|^2 \,dx + \int_\Omega \widetilde G_M(\eta_0)\,dx + \int_\Omega \widetilde \alpha_M(\eta_0) \gamma_\varepsilon(\nabla \theta_0) \,dx
    \\
    &= \frac{1}{2} \int_\Omega |\nabla \eta_0|^2 \,dx + \int_\Omega G(\eta_0)\,dx + \int_\Omega \alpha(\eta_0) \gamma_\varepsilon(\nabla \theta_0) \,dx
    \\
    &\leq \frac{1}{2}|\eta_0|_V^2 + |G(\eta_0)|_{L^1(\Omega)} + |\alpha(\eta_0)|_{L^2(\Omega)}^2 + \L^N(\Omega) + |\theta_0|_V^2.
  \end{align*}
  Hence, 
  \begin{equation*}
    C_F := \sup_{\varepsilon \in (0,1)} \F_\varepsilon(\eta_0, \theta_0) < \infty.
  \end{equation*}
  Also, from \eqref{Energy1},\eqref{n_tau}, \cref{AP}, and H\"{o}lder's inequality, it is observed that:
  \begin{align}
    & |\theta_i^\varepsilon|_V^2 \leq 2|\theta_0|_V^2 + 2 \left( \sum_{i = 1}^{n_\tau} |\theta_i^\varepsilon - \theta_{i-1}^\varepsilon|_V \right)^2 
      \label{H1_theta}
    \\
      \leq ~& 2|\theta_0|_V^2 + 2(T + 1) \sum_{i = 1}^{n_\tau} \frac{1}{\tau}|\theta_i^\varepsilon - \theta_{i-1}^\varepsilon|_V^2 
    \nonumber
      \\
      \leq~ & 2|\theta_0|_V^2 + \frac{4(T + 1)}{\delta_{\alpha} \wedge \nu^2} \left( \F_\varepsilon(\eta_0, \theta_0) + \frac{1}{2(1 \wedge \delta_\alpha)}\bigl( |u|_\sH^2 + |v|_\sH^2 + 1 \bigr)  \right)
    \nonumber
    \\
      \leq ~&  \frac{4 (C_F + 1)(T + 1)}{1 \wedge \delta_{\alpha}^2 \wedge \nu^4}\bigl( |\theta_0|_V^2 + |u|_\sH^2 + |v|_\sH^2 + 1 \bigr), \ \mbox{ for any } i = 1,2,3,\dots, n_\tau.
    \nonumber
  \end{align}

  Next, we verify the estimate \eqref{H2_eta1}. Multiplying the both side of \eqref{AP_eta} by $- \Lap (\eta_i^\varepsilon - \eta_{i-1}^\varepsilon)$ and applying Young's inequality, it can be seen that:
  \begin{align}
    &\frac{3\mu^2}{4\tau} |\Lap (\eta_i^\varepsilon - \eta_{i-1}^\varepsilon)|_H^2 \leq \frac{1}{2}\bigl( |\Lap \eta_{i-1}^\varepsilon|_H^2 - |\Lap \eta_i^\varepsilon|_H^2 \bigr) 
      \label{H2_eta2}
    \\
    &\quad + (g(\cT_M \eta_i^\varepsilon), \Lap (\eta_i^\varepsilon - \eta_{i-1}^\varepsilon))_H + (\alpha'(\cT_M \eta_i^\varepsilon) \gamma_\varepsilon(\nabla \theta_i^\varepsilon), \Lap(\eta_i^\varepsilon - \eta_{i-1}^\varepsilon))_H
    \nonumber
      \\
    &\quad + \frac{\tau}{\mu^2}|u_i|_H^2, \ \mbox{ for } i = 1,2,3,\dots, n_\tau,
    \nonumber
  \end{align}%\noeqref{H2_eta2}
  via the following calculations:
  \begin{gather*}
    %\frac{1}{\tau} (\eta_i^\varepsilon - \eta_{i-1}^\varepsilon, -\Lap (\eta_i^\varepsilon - \eta_{i-1}^\varepsilon))_H = \frac{1}{\tau}|\nabla(\eta_i^\varepsilon - \eta_{i-1}^\varepsilon)|_{[H]^N}^2 \geq 0,
    %\\
    (-\Lap \eta_i^\varepsilon, -\Lap (\eta_i^\varepsilon - \eta_{i-1}^\varepsilon))_H \geq \frac{1}{2}\bigl( |\Lap \eta_i^\varepsilon|_H^2 - |\Lap \eta_{i-1}^\varepsilon|_H^2 \bigr),
    %\\
    %\frac{\mu^2}{\tau}(-\Lap(\eta_i^\varepsilon - \eta_{i-1}^\varepsilon), -\Lap(\eta_i^\varepsilon - \eta_{i-1}^\varepsilon))_H = \frac{\mu^2}{\tau} |\Lap (\eta_i^\varepsilon - \eta_{i-1}^\varepsilon)|_H^2,
  \end{gather*}
  and
  \begin{gather*}
    (u_i, -\Lap (\eta_i^\varepsilon - \eta_{i-1}^\varepsilon))_H \leq \frac{\mu^2}{4\tau}|\Lap (\eta_i^\varepsilon - \eta_{i-1}^\varepsilon)|_H^2 + \frac{\tau}{\mu^2}|u_i|_H^2.
  \end{gather*}
  %Also, the remaining inner products are estimated as follows:
    %\begin{subequations}
    \begin{align}
    &(g(\cT_M \eta_i^\varepsilon), \Lap (\eta_i^\varepsilon - \eta_{i-1}^\varepsilon))_H 
      \label{H2_eta3a}
    \\
    &\quad \leq \frac{\tau}{\mu^2} |g|_{L^\infty(-M,M)}^2 \L^N(\Omega) + \frac{\mu^2}{4\tau}|\Lap (\eta_i^\varepsilon - \eta_{i-1}^\varepsilon)|_H^2, 
    \nonumber
    \end{align}
        and
    \begin{align}
    &(\alpha'(\cT_M \eta_i^\varepsilon)\gamma_\varepsilon(\nabla \theta), \Lap(\eta_i^\varepsilon - \eta_{i-1}^\varepsilon))_H 
      \label{H2_eta3b}
    \\
    &\quad \leq \frac{\tau}{\mu^2}|\alpha'|_{L^\infty(-M,M)}^2 \int_\Omega \bigl( \varepsilon^2 + |\nabla \theta_i^\varepsilon|^2\bigr)\,dx + \frac{\mu^2}{4\tau} |\Lap(\eta_i^\varepsilon - \eta_{i-1}^\varepsilon)|_H^2
    \nonumber
    \\
    &\quad \leq \frac{\tau}{\mu^2} |\alpha'|_{L^\infty(-M,M)}^2 \bigl( \L^N(\Omega) + |\theta_i^\varepsilon|_V^2 \bigr) + \frac{\mu^2}{4\tau}|\Lap (\eta_i^\varepsilon - \eta_{i-1}^\varepsilon)|_H^2,
    \nonumber
    \\
    &\qquad\qquad\qquad\qquad\qquad \mbox{ for } i = 1,2,3,\dots, n_\tau.
    \nonumber
  \end{align}
    %\end{subequations}
  On account of \eqref{H1_theta}--\eqref{H2_eta3b}, we infer that:
  \begin{align}
    &\frac{\mu^2}{4\tau} |\Lap (\eta_i^\varepsilon - \eta_{i-1}^\varepsilon)|_H^2 
      \label{H2_eta4}
    \\
    &\quad \leq \frac{1}{2}\bigl( |\Lap \eta_{i-1}^\varepsilon|_H^2 - |\Lap \eta_i^\varepsilon|_H^2 \bigr) + \frac{\tau \L^N(\Omega)}{\mu^2} \bigl( |g|_{L^\infty(-M,M)}^2 + |\alpha'|_{L^\infty(-M,M)} \bigr) 
    \nonumber
      \\
    &\qquad + \frac{4 \tau (C_F + 1)(T + 1)|\alpha'|_{L^\infty(-M,M)}^2}{\mu^2(1 \wedge \delta_{\alpha}^2 \wedge \nu^4)}\bigl( |\theta_0|_V^2 + |u|_\sH^2 + |v|_\sH^2 + 1 \bigr)
    \nonumber
    \\
    &\quad \leq \frac{1}{2}\bigl( |\Lap \eta_{i-1}^\varepsilon|_H^2 - |\Lap \eta_i^\varepsilon|_H^2 \bigr) + \tau \widetilde C_3\bigl( |\theta_0|_V^2 + |u|_\sH^2 + |v|_\sH^2 + 1 \bigr),
    \nonumber
    \\
    &\qquad\qquad\qquad\qquad\qquad \mbox{ for } i = 1,2,3,\dots,n_\tau,
    \nonumber
  \end{align}
  with
  \begin{equation*}
    \widetilde C_3 := \frac{4(C_F + \L^N(\Omega) + 1)(T + 1)\bigl( |g|_{L^\infty(-M,M)}^2 + |\alpha'|_{L^\infty(-M,M)}^2 +1 \bigr)}{\mu^2(1 \wedge \delta_{\alpha}^2 \wedge \nu^4)}.
  \end{equation*}
  Hence, taking the sum of \eqref{H2_eta4} with respect to $ i = 1, 2, 3, \dots, n_\tau $, one can deduce from \eqref{embb01}, \eqref{Energy1}, and \eqref{H2_eta4} that:
  \begin{align*}
    &\frac{1}{\tau} \sum_{i=1}^{n_\tau} |\eta_i^\varepsilon - \eta_{i-1}^\varepsilon|_{H^2(\Omega)}^2 \leq \frac{C_0}{\tau} \sum_{i=1}^{n_\tau} \bigl( |A_N (\eta_i^\varepsilon - \eta_{i-1}^\varepsilon)|_H^2 + |\eta_i^\varepsilon - \eta_{i-1}^\varepsilon|_H^2 \bigr)
    \\
    &\quad = \frac{C_0}{\tau} \sum_{i=1}^{n_\tau} \bigl( |\Lap (\eta_i^\varepsilon - \eta_{i-1}^\varepsilon)|_H^2 + |\eta_i^\varepsilon - \eta_{i-1}^\varepsilon|_H^2 \bigr)
    \\
    &\quad \leq \frac{2C_0}{\mu^2}|\Lap \eta_0|_H^2 + \frac{4C_0 \widetilde C_3 (T+1)}{\mu^2}\bigl( |\theta_0|_V^2 + |u|_\sH^2 + |v|_\sH^2 + 1\bigr) 
    \\
    &\quad\qquad + 4C_0 \left( \F_\varepsilon(\eta_0, \theta_0) +  \frac{1}{2(1 \wedge \delta_\alpha)} \bigl( |u|_\sH^2 + |v|_\sH^2 + 1 \bigr)  \right)
    \\
    &\quad \leq \frac{2NC_0}{\mu^2}|\eta_0|_{H^2(\Omega)}^2 + \frac{4C_0 \widetilde C_3 (T+1)}{\mu^2}\bigl( |\theta_0|_V^2 + |u|_\sH^2 + |v|_\sH^2 + 1\bigr)
    \\
    &\quad\qquad +  \frac{4C_0 (C_F + 1)}{1 \wedge \delta_\alpha} \bigl( |u|_\sH^2 + |v|_\sH^2 + 1 \bigr)
    \\
    &\quad \leq C_2 \bigl( |\eta_0|_{H^2(\Omega)}^2 + |\theta_0|_V^2 + |u|_\sH^2 + |v|_\sH^2 + 1 \bigr),
  \end{align*} 
  where
  \begin{gather*}
    C_2 := \frac{4NC_0 (T + 1)(\widetilde C_3 + C_F + 1)}{1 \wedge \mu^2 \wedge \delta_{\alpha}}.
  \end{gather*}
  
  Thus we conclude \cref{Lem_eta}.
\end{proof}

\begin{lemma}\label{Lem_theta}
  There exist a small time-step size $\tau_{**} \in (0,\tau_*)$ and a constant $C_4 > 0$ such that for any $\tau \in (0,\tau_{**})$, the following estimate holds:
  \begin{gather}
    |\theta_i^\varepsilon|_{H^2(\Omega)}^2 \leq C_4\bigl( |\eta_0|_{H^2(\Omega)}^2 + |\theta_0|_{H^2(\Omega)}^2 + |u|_\sH^2 + |v|_\sH^2 + 1 \bigr)^2, 
      \label{H2_theta1}
    \\
    \mbox{for } i= 1,2,3,\dots, n_\tau.
      \nonumber
  \end{gather}
\end{lemma}

We can prove this \cref{Lem_theta} by using the following technical lemma, obtained in \cite{aiki2023class}.

\begin{lemma} \label{SD_part}
  (cf. \cite[Lemma 3.2]{aiki2023class}) Let us fix $ \varepsilon > 0$, $w \in W_0 $, and $ \alpha^\circ \in L^\infty(\Omega) \cap V$. Then, for any $L \geq |\alpha^\circ|_{L^\infty(\Omega)}$, there exists a constant $C_5(L) > 0$, depending only on $L$, and being independent of $\varepsilon$ and $w$, such that:
  \begin{equation*}
    \bigl( \diver (\alpha^\circ \nabla \gamma_\varepsilon(\nabla w) ), \Lap w \bigr)_H \geq -|\nabla^2 w|_{[H]^{N \times N}}^2 - C_5(L) (|\alpha^\circ|_V^2 + 1)(|w|_V^2 + 1).
  \end{equation*}
\end{lemma}

\begin{proof}[Proof of \cref{Lem_theta}]
  First, we note that \cref{Lem_eta} leads to the boundedness of $\{ \eta_i^\varepsilon \}_{i = 0}^{n_\tau}$ in $H^2(\Omega)$, with the following estimate:
  \begin{align}
    |\eta_i^\varepsilon|_{H^2(\Omega)}^2 &\leq 2|\eta_0|_{H^2(\Omega)}^2 + 2 \left( \sum_{i=1}^{n_\tau} |\eta_i^\varepsilon - \eta_{i-1}^\varepsilon|_{H^2(\Omega)} \right)^2 
      \label{H2_eta5}
    \\
    &\leq 2|\eta_0|_{H^2(\Omega)}^2 + 2 (T+1) \sum_{i=1}^{n_\tau} \frac{1}{\tau}|\eta_i^\varepsilon - \eta_{i-1}^\varepsilon|_{H^2(\Omega)}^2
    \nonumber
      \\
    &\leq 2(T + 1)(C_2 + 1)\bigl( |\eta_0|_{H^2(\Omega)}^2 + |\theta_0|_V^2 + |u|_\sH^2 + |v|_\sH^2 + 1 \bigr),
    \nonumber
    \\
    &\qquad\qquad\qquad \ \mbox{ for } i = 1,2,3,\dots, n_\tau.
    \nonumber
  \end{align}
  Moreover, by \eqref{H2_eta5} and continuous embedding from $H^2(\Omega)$ to $L^\infty(\Omega)$ under $N \leq 3$, we see that $\widetilde \alpha_M(\eta_i^\varepsilon) \in L^\infty(\Omega) \cap V$ for any $i = 1,2,3,\dots, n_\tau$, with the following estimates hold:
  \begin{align}
      |\widetilde \alpha_M(\eta_i^\varepsilon)|_{L^\infty(\Omega)}^2 & \leq 2\alpha(0)^2 + 2|\alpha'|_{L^\infty(-M,M)}^2 |\eta_i^\varepsilon|_{L^\infty(\Omega)}^2 
      \label{ken01}
    \\
    & \leq 2\alpha(0)^2 + 2(C_{H^2}^{L^\infty})^2 |\alpha'|_{L^\infty(-M,M)}^2 |\eta_i^\varepsilon|_{H^2(\Omega)}^2,
      \nonumber
  \end{align}
  and
  \begin{gather}
    |\widetilde \alpha_M(\eta_i^\varepsilon)|_V^2 \leq \L^N(\Omega)|\widetilde \alpha_M(\eta_i^\varepsilon)|_{L^\infty(\Omega)}^2 + |\alpha'|_{L^\infty(-M,M)}^2 |\nabla \eta_i^\varepsilon|_{[H]^N}^2 
      \label{dm01}
    \\
    \leq 2 \L^N(\Omega) \bigl( \alpha(0)^2 + (C_{H^2}^{L^\infty} )^2 |\alpha'|_{L^\infty(-M,M)}^2 |\eta_i^\varepsilon|_{H^2(\Omega)}^2 \bigr) + |\alpha'|_{L^\infty(-M,M)}^2 |\eta_i^\varepsilon|_{H^2(\Omega)}^2
      \nonumber
    \\
    \leq \widetilde C_6\bigl( |\eta_i^\varepsilon|_{H^2(\Omega)}^2 + 1 \bigr), \ \mbox{ for } i = 1,2,3,\dots,n_\tau,
      \nonumber
  \end{gather}
  where $C_{H^2}^{L^\infty}$ is a constant of the embedding from $H^2(\Omega)$ to $L^\infty(\Omega)$, and
  \begin{equation*}
    \widetilde C_6 := 2 \bigl( \L^N(\Omega) \alpha(0)^2 + \L^N(\Omega) ( C_{H^2}^{L^\infty} )^2 + 1 \bigr)\bigl( |\alpha'|_{L^\infty(-M,M)}^2 + 1 \bigr),
  \end{equation*}

  Next, we verify the inequality \eqref{H2_theta1}. Let us consider to multiply the both sides of \eqref{AP_theta} by $-\Lap \theta_i^\varepsilon$. 

    By applying Young's inequality, we have:
  \begin{gather}
    \frac{\nu^2}{\tau}\bigl( -\Lap(\theta_i^\varepsilon - \theta_{i-1}^\varepsilon), -\Lap \theta_i^\varepsilon \bigr)_H \geq \frac{\nu^2}{2\tau} \bigl( |\Lap \theta_i^\varepsilon|_H^2 - |\Lap \theta_{i-1}^\varepsilon|_H^2 \bigr), 
      \label{H2_theta2}
  \end{gather}
  and
  \begin{gather}
    (v_i, -\Lap \theta_i^\varepsilon)_H \leq \frac{1}{2}|\Lap \theta_i^\varepsilon|_H^2 + \frac{1}{2}|v_i|_H^2, \ \mbox{ for } i = 1,2,3,\dots, n_\tau. 
      \label{H2_theta22}
  \end{gather}%\noeqref{H2_theta22}
  Moreover, from \eqref{Energy1} and (A3), we see that:
  \begin{align}
    &\frac{1}{\tau} \bigl( \alpha_0(\cT_M \eta_{i-1}^\varepsilon) (\theta_i^\varepsilon - \theta_{i-1}^\varepsilon), -\Lap \theta_i^\varepsilon \bigr)_H 
      \label{H2_theta3}
    \\
    &\qquad \geq - \frac{1}{2\tau^2} |\alpha_0|_{L^\infty(-M,M)}^2 |\theta_i^\varepsilon - \theta_{i-1}^\varepsilon|_H^2 - \frac{1}{2}|\Lap \theta_i^\varepsilon|_H^2,
    \nonumber
      \\
    &\qquad \geq \frac{|\alpha_0|_{L^\infty(-M,M)}^2}{2\delta_{\alpha}} \cdot \frac{1}{\tau}(\F_\varepsilon(\eta_i^\varepsilon, \theta_i^\varepsilon) - \F_\varepsilon(\eta_{i-1}^\varepsilon, \theta_{i-1}^\varepsilon)) - \frac{1}{2}|\Lap \theta_i^\varepsilon|_H^2,
    \nonumber
    \\
    &\qquad\qquad\qquad\qquad \mbox{ for } i = 1,2,3,\dots, n_\tau.
    \nonumber
  \end{align}%\noeqref{H2_theta3}
  Using \eqref{embb01}, \eqref{H2_eta5}--\eqref{dm01}, and applying \cref{SD_part} to the case that:
  \begin{align*}
    \alpha^\circ = \widetilde \alpha_M(\eta_{i-1}^\varepsilon) \mbox{ and } w = \theta_i^\varepsilon, \mbox{ for each } i \in \{ 1,2, \ldots, n_\tau \},
  \end{align*}
  and 
  \begin{align*}
      L & = 2 \alpha(0)^2 + 2 (C_{H^2}^{L^\infty})^2 |\alpha'|_{L^\infty(-M,M)}^2 \cdot 
      \\
      & \cdot  2(T + 1)(C_2 + 1)\bigl( |\eta_0|_{H^2(\Omega)}^2 + |\theta_0|_V^2 + |u|_\sH^2 + |v|_\sH^2 + 1 \bigr),
  \end{align*}
  it is observed that:
  \begin{align}
    &\quad \bigl( \diver (\widetilde \alpha_M(\eta_{i-1}^\varepsilon) \nabla \gamma_\varepsilon(\nabla \theta_i^\varepsilon)), \Lap \theta_i^\varepsilon \bigr)_H 
      \label{H2_theta4}
    \\
    &\geq -|\nabla^2 \theta_i^\varepsilon|_{[H]^{N \times N}}^2 - C_5(L) (|\widetilde \alpha_M(\eta_{i-1}^\varepsilon)|_V^2 + 1)(|\theta_i^\varepsilon|_V^2 + 1)
    \nonumber
      \\
    &\geq -C_0 \bigl( |A_N \theta_i^\varepsilon|_H^2 + |\theta_i^\varepsilon|_H^2 \bigr)
    \nonumber
    \\
    &\qquad- C_5(L)\bigl( \widetilde C_6\bigl( |\eta_{i-1}^\varepsilon|_{H^2(\Omega)}^2 + 1 \bigr) + 1 \bigr)\bigl( |\theta_i^\varepsilon|_V^2 + 1 \bigr)
    \nonumber
    \\
    &\geq -C_0 |\Lap \theta_i^\varepsilon|_H^2 - (C_5(L) + C_0)(\widetilde C_6 + 1) \bigl( |\theta_i^\varepsilon|_V^2 + 1 \bigr)\bigl( |\eta_{i-1}^\varepsilon|_{H^2(\Omega)}^2 + 1 \bigr),
    \nonumber
    \\
    &\qquad\qquad\qquad\qquad \mbox{ for } i = 1,2,3,\dots, n_\tau.
    \nonumber
  \end{align}

  Now, by using \eqref{H2_theta2}--\eqref{H2_theta4}, we will obtain that:
  \begin{equation}
    \frac{1}{\tau}(X_i - X_{i-1}) \leq \widetilde C_7(X_i + F_i), \ \mbox{ for } i = 1,2,3,\dots,n_\tau, 
      \label{H2_theta5}
  \end{equation}
  with 
  \begin{gather*}
      \begin{cases}
          \ds X_i := \nu^2 |\Lap \theta_i^\varepsilon|_H^2 + \frac{|\alpha_0|_{L^\infty(-M,M)}^2}{\delta_{\alpha}} \F_\varepsilon(\eta_i^\varepsilon, \theta_i^\varepsilon), %\ \mbox{ for } i = 0,1,2,\dots, n_\tau,
          \\[2ex]
    \ds F_i := \bigl( |\theta_i^\varepsilon|_V^2 + |v_i|_H^2 + 1 \bigr)\bigl( |\eta_{i-1}^\varepsilon|_{H^2(\Omega)}^2 + 1 \bigr), 
      \end{cases} \mbox{ for } i= 1,2,3,\dots, n_\tau,
  \end{gather*}
  and
  \begin{equation*}
    \widetilde C_7 := \frac{4(C_5(L) + C_0 + 1)(\widetilde C_6 + 1)}{1 \wedge \nu^2} \geq 4.
  \end{equation*}
  Here, let us take $\tau_{**} \in (0,\tau_*)$ satisfying:
  \begin{equation*}
    \tau_{**} < \min \left\{ \tau_*, \frac{1}{2 \widetilde C_7} \right\}, \mbox{ and in particular, } 1- \tau_{**} \widetilde C_7 > \frac{1}{2}.
  \end{equation*}
  Then, applying the discrete version of Gronwall's lemma (cf. \cite[Section 3.1]{emmrich1999discrete}) to \eqref{H2_theta5}, one can see from \eqref{H1_theta}, \eqref{H2_eta5}, and \eqref{H2_theta5} that:
  \begin{align}
    &\quad X_i \leq \widetilde C_7 e^{2 \widetilde C_7(T + 1)}\left( X_0 + \tau\sum_{i=1}^{n_\tau} F_i \right) 
      \label{H2_theta6}
    \\
    &\leq \widetilde C_7e^{2 \widetilde C_7(T + 1)} \left( \nu^2 |\Lap \theta_0|_H^2 + \frac{|\alpha_0|_{L^\infty(-M,M)}^2 C_F}{\delta_{\alpha}} \right)
    \nonumber
      \\
    &\,\quad + \frac{4\widetilde C_7e^{2 \widetilde C_7(T + 1)}(C_F + 3)(T + 1)^2}{1 \wedge \delta_{\alpha}^2 \wedge \nu^4} \bigl( |\theta_0|_V^2 + |u|_\sH^2 + |v|_\sH^2 + 1 \bigr) \cdot 
    \nonumber
    \\
    &\qquad\quad \cdot 2(T + 1)(C_2 + 2)\bigl( |\eta_0|_{H^2(\Omega)}^2 + |\theta_0|_V^2 + |u|_\sH^2 + |v|_\sH^2 + 1 \bigr)
    \nonumber
    \\
    &\leq \widetilde C_8\bigl( |\eta_0|_{H^2(\Omega)}^2 + |\theta_0|_{H^2(\Omega)}^2 + |u|_\sH^2 + |v|_\sH^2 + 1 \bigr)^2, \ \mbox{ for } i = 1,2,3,\dots, n_\tau,
    \nonumber
  \end{align}
  where
  \begin{equation*}
    \widetilde C_8 := \frac{8 \widetilde C_7 e^{2 \widetilde C_7(T + 1)}(C_F + 3)(T + 1)^3 (C_2 + N \nu^2 + |\alpha_0|_{L^\infty(-M,M)} + 2)}{1 \wedge \delta_{\alpha}^2 \wedge \nu^4}.
  \end{equation*}
  In the light of \eqref{embb01} \eqref{H1_theta}, and \eqref{H2_theta6}, we arrive at:
  \begin{align*}
    &\quad |\theta_i^\varepsilon|_{H^2(\Omega)}^2 \leq C_0 \bigl( |A_N \theta_i^\varepsilon|_H^2 + |\theta_i^\varepsilon|_H^2 \bigr) = C_0 \bigl( |\Lap \theta_i^\varepsilon|_H^2 + |\theta_i^\varepsilon|_H^2 \bigr)
    \\
    &\leq \frac{C_0}{\nu^2} X_i + \frac{4C_0(C_F + 1)(T + 1)}{1 \wedge \delta_{\alpha}^2 \wedge \nu^4}\bigl( |\theta_0|_V^2 + |u|_\sH^2 + |v|_\sH^2 + 1 \bigr)
    \\
    &\leq \frac{4C_0(C_F + \widetilde C_8 + 1)(T + 1)}{1 \wedge \delta_{\alpha}^2 \wedge \nu^4} \bigl( |\eta_0|_{H^2(\Omega)}^2 + |\theta_0|_{H^2(\Omega)}^2 + |u|_\sH^2 + |v|_\sH^2 + 1 \bigr)^2,
    \\
    &\qquad\qquad\qquad\qquad\qquad\mbox{ for } i = 1,2,3,\dots,n_\tau.
  \end{align*}
  
  Thus we conclude \cref{Lem_theta} with the constant:
  \begin{equation*}
    C_4 = \frac{4C_0(C_F + \widetilde C_8 + 1)(T + 1)}{1 \wedge \delta_{\alpha}^2 \wedge \nu^4}.
  \end{equation*}
\end{proof}

Next, we confirm the comparison principle for single pseudo-parabolic equation, which will play a key-role in the $L^\infty$-estimate of the component $\eta$.

\begin{lemma} \label{Lem_CPeta}
  We assume that $\eta^1, \eta^2 \in W^{1,2}(0,T;W_0)$, $\eta_0^1, \eta_0^2 \in W_0$, $\widetilde \theta \in \sV$, $\widetilde u \in \sH$, and
  \begin{equation}
    \left\{ \begin{aligned}
      &(-1)^{i-1} \left( \partial_t \eta^i - \Lap (\eta^i + \mu^2 \partial_t \eta^i ) + g(\cT_M \eta^i) + \alpha'(\cT_M \eta^i)|\nabla \widetilde \theta| \right) 
        \\
        & \qquad \leq \widetilde u, \mbox{ a.e. in } Q,
      \\
      &\eta^i(0) = \eta_0^i, \mbox{ a.e. in } \Omega.
    \end{aligned} \right. \label{CP_eta1}
  \end{equation}
  Then, there exists a constant $C_9 > 0$ such that:
  \begin{equation*}
    \bigl| [\eta^1 - \eta^2]^+(t) \bigr|_V^2 \leq C_9 \bigl| [\eta_0^1 - \eta_0^2]^+ \bigr|_V^2, \mbox{ for any } t \in [0,T].
  \end{equation*}
\end{lemma}

\begin{proof}
  Taking the difference of two inequality \eqref{CP_eta1} for $\eta^i$, $i = 1,2$, and multiplying the both sides by $[\eta^1 - \eta^2]^+(t)$, we see that:
  \begin{align}
    &\frac{1}{2}\frac{d}{dt} \bigl| [\eta^1 - \eta^2]^+(t) \bigr|_H^2 + \frac{\mu^2}{2} \frac{d}{dt} \bigl| \nabla[\eta^1 - \eta^2]^+(t) \bigr|_{[H]^N}^2 \label{CP_eta3}
    \\
    &\qquad + \bigl| \nabla [\eta^1 - \eta^2]^+(t) \bigr|_{[H]^N}^2 + (g(\cT_M \eta^1(t)) - g(\cT_M \eta^2(t)), [\eta^1 - \eta^2]^+(t))_H 
    \nonumber
      %\\
    %&\qquad + \bigl((\alpha'(\cT_M \eta^1) - \alpha'(\cT_M \eta^2)(t)) |\nabla \theta(t)|, [\eta^1 - \eta^2]^+(t) \bigr) 
     % \nonumber
      \\
      &\qquad + \int_\Omega\bigl((\alpha'(\cT_M \eta^1(t)) - \alpha'(\cT_M \eta^2(t))\bigr) |\nabla \theta(t)|, [\eta^1 - \eta^2]^+(t) \, dx 
    \nonumber
      \\
      & \leq 0, \ \mbox{ for a.e. } t \in (0,T).
    \nonumber
  \end{align}
  Here, from the assumption (A1), it is deduced that:
  \begin{align}
    &(g(\cT_M \eta^1(t)) - g(\cT_M \eta^2(t)), [\eta^1 - \eta^2]^+(t))_H \notag
    \\
    &\qquad \geq - |g'|_{L^\infty(-M,M)} \bigl| [\eta^1 - \eta^2]^+(t) \bigr|_H^2, \ \mbox{ for a.e. } t \in (0,T). \label{CP_eta4}
  \end{align}
  Also, by the monotonicity of $\alpha' \circ \cT_M$, we can say that:
  \begin{equation*}
    \bigl( \alpha'(\cT_M \eta^1) - \alpha'(\cT_M \eta^2)\bigr) [\eta^1 - \eta^2]^+ \geq 0, \ \mbox{ a.e. in } Q. 
  \end{equation*}
  Hence one can see that:
  \begin{gather}
    \int_\Omega\bigl((\alpha'(\cT_M \eta^1(t)) - \alpha'(\cT_M \eta^2(t))\bigr) |\nabla \theta(t)|, [\eta^1 - \eta^2]^+(t) \, dx \notag
      %\bigl((\alpha'(\cT_M \eta^1) - \alpha'(\cT_M \eta^2)(t)) |\nabla \theta(t)|, [\eta^1 - \eta^2]^+(t) \bigr) 
      \\
      \geq 0, \ \mbox{ for a.e. } t \in (0,T). \label{CP_eta5}
  \end{gather}
  
  Now, in the light of \eqref{CP_eta3}--\eqref{CP_eta5}, it is deduced that:
  \begin{align*}
    &\frac{d}{dt} \left( \bigl| [\eta^1 - \eta^2]^+(t) \bigr|_H^2 + \mu^2 \bigl| \nabla[\eta^1 - \eta^2]^+(t) \bigr|_{[H]^N}^2 \right) 
    \\
    &\qquad\leq 2|g'|_{L^\infty(-M,M)} \bigl| [\eta^1 - \eta^2]^+(t)\bigr|_H^2, \ \mbox{ for a.e. } t \in (0,T).
  \end{align*}
  Applying Gronwall's inequality, we arrive at:
  \begin{align}
    &\bigl| [\eta^1 - \eta^2]^+(t) \bigr|_H^2 + \mu^2 \bigl| \nabla[\eta^1 - \eta^2]^+(t) \bigr|_{[H]^N}^2 \label{CP_eta6} 
    \\
    &\ \leq e^{2T |g'|_{L^\infty(-M,M)}} \left(  \bigl| [\eta_0^1 - \eta_0^2]^+ \bigr|_H^2 + \mu^2\bigl| \nabla [\eta_0^1 - \eta_0^2]^+ \bigr|_{[H]^N}^2 \right), \notag
    \\
    &\qquad\qquad\qquad\qquad \mbox{ for any } t \in [0,T]. \notag
  \end{align}
  \eqref{CP_eta6} finishes the proof of \cref{Lem_CPeta} with the constant:
  \begin{equation*}
    C_9 := \frac{1 + \mu^2}{1 \wedge \mu^2} e^{2T |g'|_{L^\infty(-M,M)}}.
  \end{equation*}
\end{proof}

\begin{proof}[The proof of \cref{mainThm1}]
    Let us take $\varepsilon \in (0,1)$, $\tau \in (0,\tau_{**})$ where $\tau_{**}$ is given in \cref{Lem_theta}. As a consequence of \cref{Lem_eta,Lem_theta,AP}, the following boundedness are derived:
  \begin{description}
      \item[~~$\bullet$] $\{ [\eta^\varepsilon]_\tau \, | \, \varepsilon \in (0,1), \tau \in (0,\tau_{**}) \}$ is bounded in $W^{1,2}(0,T;W_0)$,
    \item[~~$\bullet$] $\{ [\overline{\eta}^\varepsilon]_\tau \, | \, \varepsilon \in (0,1), \tau \in (0,\tau_{**}) \}$, $\{ [\underline{\eta}^\varepsilon]_\tau \, | \, \varepsilon \in (0,1), \tau \in (0,\tau_{**}) \}$ is bounded in $L^\infty(0,T;W_0)$,
    \item[~~$\bullet$] $\{ [\theta^\varepsilon]_\tau \, | \, \varepsilon \in (0,1), \tau \in (0,\tau_{**}) \}$ is bounded in $L^\infty(0,T;W_0)$ and in $W^{1,2}(0,T;V)$,
    \item[~~$\bullet$] $\{ [\overline{\theta}^\varepsilon]_\tau \, | \, \varepsilon \in (0,1), \tau \in (0,\tau_{**}) \}$, $\{ [\underline{\theta}^\varepsilon]_\tau \, | \, \varepsilon \in (0,1), \tau \in (0,\tau_{**}) \}$ is bounded in $L^\infty(0,T;W_0)$.
  \end{description}
  Therefore, by applying Aubin's type compactness theory (cf. \cite[Corollary 4]{MR0916688}), we can find sequences $\{ \varepsilon_n \}_{n = 1}^\infty \subset (0,1)$, $\{ \tau_n \}_{n = 1}^\infty \subset (0,\tau_{**})$ and a pair of functions $[\eta, \theta] \in [\sH]^2$ such that $\varepsilon_n \searrow 0$ and $\tau_n \searrow 0$, as $n \to \infty$, we obtain the following convergences as $n \to \infty$:
  \begin{equation}
    \left\{ \begin{aligned}
      &\eta_n := [\eta^{\varepsilon_n}]_{\tau_n} \to \eta \mbox{ in } C([0,T];V), \mbox{ and weakly in } W^{1,2}(0,T;W_0),
      \\
      &\theta_n := [\theta^{\varepsilon_n}]_{\tau_n} \to \theta \mbox{ in } C([0,T];V), \mbox{ weakly in } W^{1,2}(0,T;V),
      \\
      &\qquad \mbox{and weakly-$*$ in } L^\infty(0,T;W_0). 
    \end{aligned}\right. \label{conv_1}
  \end{equation}
  Besides, having in mind:
  \begin{equation*}
    \left\{ \begin{aligned}
      &\max\left\{ |[\overline{\eta}^{\varepsilon_n}]_{\tau_n} - \eta_n|_V, \, |[\underline{\eta}^{\varepsilon_n}]_{\tau_n} - \eta_n|_V \right\} \leq \int_{\Delta_{i,\tau}} |\partial_t \eta_n|_V \,dt < \infty,
      \\
      &\max\left\{ |[\overline{\theta}^{\varepsilon_n}]_{\tau_n} - \theta_n|_V, \, |[\underline{\theta}^{\varepsilon_n}]_{\tau_n} - \theta_n|_V \right\} \leq \int_{\Delta_{i,\tau}} |\partial_t \theta_n|_V \,dt < \infty,
    \end{aligned} \right. \label{conv_2}
  \end{equation*}
  we can derive that:
  \begin{equation}
    \left\{ \begin{aligned}
      &\overline{\eta}_n := [\overline{\eta}^{\varepsilon_n}]_{\tau_n} \to \eta \mbox{ and } \underline{\eta}_n := [\underline{\eta}^{\varepsilon_n}]_{\tau_n} \to \eta \mbox{ in } L^\infty(0,T;V), \mbox{ and }
      \\
      &\mbox{weakly-$*$ in } L^\infty(0,T;W_0),
      \\
      &\overline{\theta}_n := [\overline{\theta}^{\varepsilon_n}]_{\tau_n} \to \theta \mbox{ and } \underline{\theta}_n := [\underline{\theta}^{\varepsilon_n}]_{\tau_n} \to \theta \mbox{ in } L^\infty(0,T;V), \mbox{ and }
      \\
      &\mbox{weakly-$*$ in } L^\infty(0,T;W_0).
    \end{aligned}\right. \label{conv_3}
  \end{equation}

  Now, we verify that the limiting pair $[\eta, \theta]$ satisfies (S0)--(S3). Let us take an arbitrary open interval $I \subset (0,T)$. Then, in the light of \eqref{AP_eta}, \eqref{AP_theta}, and the convexity of $\gamma_\varepsilon$, the sequences as in \eqref{conv_1} and \eqref{conv_3} should fulfill the following two variational forms:
  \begin{align}
    &\int_I (\partial_t \eta_n(t), \varphi)\,dt + \int_I (\nabla (\overline{\eta}_n + \mu^2 \partial_t \eta_n)(t), \nabla \varphi)_{[H]^N} \,dt \label{conv_4}
    \\
    &\quad + \int_I (g(\cT_{M} \,\overline{\eta}_n(t)), + \alpha'(\cT_{M} \,\overline{\eta}_n(t))\gamma_\varepsilon(\nabla \overline{\theta}_n(t)), \varphi)_H \,dt = \int_I ([\overline{u}]_{\tau_n}(t), \varphi)_H \,dt,\notag
    \\
    &\qquad\qquad\qquad\mbox{ for all $\varphi \in V$, and $n = 1,2,3,\dots, n_\tau$,} \notag
  \end{align}
  and
  \begin{align}
    &\int_I \bigl((\alpha_0(\cT_{M} \underline{\eta}_n) \partial_t \theta_n)(t), \overline{\theta}_n(t) - \psi \bigr)_H \,dt +\int_I \int_\Omega \alpha(\underline{\eta}_n(t)) \gamma_\varepsilon(\nabla \overline{\theta}_n(t)) \,dxdt \notag
    \\
    &\quad + \nu^2 \int_I (\nabla \partial_t \theta_n(t), \nabla (\overline{\theta}_n(t) - \psi))_{[H]^N}\,dt \label{conv_5}
    \\
    &\quad \leq \int_I \int_\Omega \alpha(\underline{\eta}_n(t)) \gamma_\varepsilon(\nabla \psi)\,dxdt + \int_I ([\overline{v}]_{\tau_n}(t), \overline{\theta}_n(t) - \psi)_H \,dt,\notag
    \\
    &\qquad\qquad\qquad\mbox{ for all $\psi \in V$, and $n = 1,2,3,\dots, n_\tau$.}\notag
  \end{align}
  On this basis, having in mind \eqref{conv_1}, (A1), (A2), and the fact:
  \begin{equation*}
    \gamma_\varepsilon \to |\cdot| \mbox{ uniformly on } \R^N, \mbox{ as } \varepsilon \to 0, 
  \end{equation*}
  letting $n \to \infty$ in \eqref{conv_4} and \eqref{conv_5} yields that:
  \begin{align}
    &\int_I (\partial_t \eta(t), \varphi)\,dt + \int_I (\nabla (\eta + \mu^2 \partial_t \eta)(t), \nabla \varphi)_{[H]^N} \,dt \label{conv_6}
    \\
    &\quad + \int_I (g(\cT_{M} \,\eta(t)) + \alpha'(\cT_{M} \,\eta(t))|\nabla \theta(t)|, \varphi)_H \,dt = \int_I (u(t), \varphi)_H \,dt, \notag
    \\
    &\qquad\qquad\qquad\qquad\qquad\qquad\mbox{ for any } \varphi \in V, \notag
  \end{align}
  and
  \begin{align*}
    &\int_I \bigl((\alpha_0(\cT_{M} \eta) \partial_t \theta)(t), \theta(t) - \psi \bigr)_H \,dt +\int_I \int_\Omega \alpha(\eta(t)) |\nabla \theta(t)| \,dxdt \label{conv_7}
    \\
    &\quad + \nu^2 \int_I (\nabla \partial_t \theta(t), \nabla (\theta(t) - \psi))_{[H]^N}\,dt 
    \\
    &\quad \leq \int_I \int_\Omega \alpha(\eta(t)) |\nabla \psi| \,dxdt + \int_I (v(t), \theta(t) - \psi)_H \,dt, \ \mbox{ for any } \psi \in V,
  \end{align*}
  respectively. Since $I \subset (0,T)$ is arbitrary, $[\eta, \theta]$ should satisfy (S1) and (S2).
  
  Next, let us verify $\eta \in L^\infty(Q)$. By \eqref{ken02}, the following inequalities can be obtained:
  \begin{gather*}
    \left\{ \begin{aligned}
      &\partial_t M - \Lap (M + \mu^2 \partial_t M) + g(M) + \alpha'(M)|\nabla \theta(t)| \geq u,
      \\
      &\partial_t (-M) - \Lap ((-M) + \mu^2 \partial_t (-M)) + g(-M) + \alpha'(-M)|\nabla \theta(t)| \leq u,
    \end{aligned} \right.
    \\
    \mbox{a.e. in } Q.
  \end{gather*}
  Hence, applying \cref{Lem_CPeta} to the case when
  \begin{equation*}
    [\eta^1, \eta^2, \widetilde \theta, \widetilde u] = [\eta, M, \theta, u], ~~ [\eta_0^1, \eta_0^2] = [\eta_0, M],
  \end{equation*}
  and
  \begin{equation*}
    [\eta^1, \eta^2, \widetilde \theta, \widetilde u] = [-M, \eta, \theta, u], ~~ [\eta_0^1, \eta_0^2] = [-M, \eta_0],
  \end{equation*}
  we arrive at:
  \begin{align}
    &\left\{ \begin{aligned}
      &\bigl| [\eta - M]^+(t) \bigr|_V^2 \leq C_9 \bigl| [\eta_0 - M]^+ \bigr|_V^2 = 0,
      \\
      &\bigl| [- M - \eta]^+(t) \bigr|_V^2 \leq C_9 \bigl| [- M - \eta_0]^+ \bigr|_V^2 = 0,
    \end{aligned} \right.
     \mbox{ for any } t \in [0,T],
  \end{align}
  respectively. This implies that:
  \begin{equation*}
    |\eta(t)|_{L^\infty(\Omega)} \leq M, \ \mbox{ for any } t \in [0,T].
  \end{equation*}

  Finally, due to \eqref{conv_1}, the condition (S3) as in \cref{mainThm1} is immediately confirmed as follows:
  \begin{equation*}
    \eta(0) = \lim_{n \to \infty} \eta_n(0) = \eta_0 \mbox{ and } \theta(0) = \lim_{n \to \infty} \theta_n(0) = \theta_0 \mbox{ in } H.
  \end{equation*}

  Thus, we conclude that $[\eta, \theta]$ is a solution to (S).
  
\end{proof}

\subsection{Proof of \cref{mainThm2}}
Let $[\eta^k, \theta^k], \, k = 1,2$, be the solutions to (S) corresponding to initial values $\eta_0^k, \, \theta_0^k$ and forcings $u^k, \, v^k$, $k = 1,2$. Let us set
\begin{equation}
  M_0 := |\eta^1|_{L^\infty(Q)} \vee |\eta^2|_{L^\infty(Q)}, \mbox{ and } \delta_*(\nu) := 1 \wedge \delta_\alpha \wedge \nu^2, \label{ken03}
\end{equation}
and take the difference between the variational formulas for $\eta^k, \, k = 1,2$, and put $\varphi := (\eta^1 - \eta^2)(t)$. Then, by using (A1), the monotonicity of $\alpha'$, and Young's inequality, we see that:
\begin{align}
  &\frac{1}{2}\frac{d}{dt}\bigl( |(\eta^1 - \eta^2)(t)|_H^2 + \mu^2 |\nabla (\eta^1 - \eta^2)(t)|_{[H]^N}^2 \bigr) 
    \label{uni_1}
    \\
    & \quad\qquad- |g|_{L^\infty(-M_0,M_0)}|(\eta^1 - \eta^2)(t)|_H^2
  \nonumber
    \\
    &\quad \leq \int_\Omega \alpha'(\eta^1(t)) (\eta^2 - \eta^1)(t) |\nabla \theta^1(t)| \, dx
    \nonumber
    \\
    &\quad\qquad + \int_\Omega \alpha'(\eta^2(t)) (\eta^1 - \eta^2)(t) |\nabla \theta^2(t)| \, dx
    \nonumber
  \\
    &\quad\qquad + \bigl( (\eta^1 - \eta^2)(t), (u^1 - u^2)(t) \bigr)_H
  \nonumber
  \\
    &\quad \leq \int_\Omega |\alpha(\eta^1(t)) -\alpha(\eta^2(t))||\nabla (\theta^1 -\theta^2)(t)|\,dx
  \nonumber
  \\
  &\quad\qquad + |(\eta^1 - \eta^2)(t)|_H|(u^1 - u^2)(t)|_H
  \nonumber
  \\
  &\quad \leq \frac{|\alpha'|_{L^\infty(-M_0,M_0)}}{2}\left( |(\eta^1 - \eta^2)(t)|_H^2 + |\nabla(\theta^1 - \theta^2)(t)|_{[H]^N}^2 \right)
  \nonumber
  \\
  &\quad\qquad + \frac{1}{2}|(\eta^1 - \eta^2)(t)|_H^2 + \frac{1}{2}|(u^1 - u^2)(t)|_H^2, \ \mbox{ for a.e. }t > 0.
  \nonumber
\end{align}
On the other hand, by putting $\psi = \theta^2$ in the variational inequality for $\theta^1$, and $\psi = \theta^1$ in the one for $\theta^2$, and by taking the sum of two inequalities, we have:
\begin{align}
  &\bigl( \alpha_0(\eta^1)\partial_t \theta^1(t) - \alpha_0(\eta^2)\partial_t \theta^2(t), (\theta^1 - \theta^2)(t) \bigr)_H 
    \label{uni_2}
    \\
    & \quad \qquad + \frac{1}{2}\frac{d}{dt} \bigl( \nu^2|\nabla (\theta^1 - \theta^2)(t)|_{[H]^N}^2 \bigr)
  \nonumber
    \\
    &\quad \leq \int_\Omega \alpha(\eta^1(t)) \bigl( |\nabla \theta^2(t)| - |\nabla \theta^1(t)| \bigr) \, dx 
    \nonumber
    \\
    &\quad\qquad + \int_\Omega \alpha(\eta^2(t)) \bigl( |\nabla \theta^1(t)| - |\nabla \theta^2(t)| \bigr) \, dx 
  \nonumber
  \\
    &\quad\qquad+ \bigl((\theta^1 - \theta^2)(t),(v^1 - v^2)(t) \bigr)_H 
  \nonumber
    \\
    &\quad \leq \int_\Omega |\alpha(\eta^1(t)) -\alpha(\eta^2(t))|\nabla (\theta^1 - \theta^2)(t)|\,dx
  \nonumber
  \\
  &\quad\qquad+ |(\theta^1 - \theta^2)(t)|_H|(v^1 - v^2)(t)|_H 
  \nonumber
  \\
  &\quad \leq \frac{|\alpha'|_{L^\infty(-M_0,M_0)}}{2}\left( |(\eta^1 - \eta^2)(t)|_H^2 + |\nabla(\theta^1 - \theta^2)(t)|_{[H]^N}^2 \right)
  \nonumber
  \\
  &\quad\qquad + \frac{1}{2}|(\theta^1 - \theta^2)(t)|_H^2 + \frac{1}{2}|(v^1 - v^2)(t)|_H^2, \ \mbox{ a.e. }t > 0.
  \nonumber
\end{align}
Here, we can compute the first term in \eqref{uni_2} as follows:
\begin{align}
  &\quad\bigl( \alpha_0(\eta^1(t))\partial_t \theta^1(t) - \alpha_0(\eta^2(t))\partial_t \theta^2(t), (\theta^1 - \theta^2)(t) \bigr)_H \label{dm04}
  \\
    &= \frac{1}{2}\frac{d}{dt} |\sqrt{\alpha_0(\eta^1(t))}(\theta^1 - \theta^2)(t)|_H^2 - \frac{1}{2}\int_\Omega \alpha_0'(\eta^1(t)) \partial_t \eta^1(t) |(\theta^1 - \theta^2)(t)|^2 \,dx \notag
  \\
    &\quad + \int_\Omega (\alpha_0(\eta^1(t)) - \alpha_0(\eta^2(t)))\partial_t \theta^2(t) (\theta^1 - \theta^2)(t) \,dx. \notag
\end{align}
Also, by using \eqref{ken03}, the continuous embedding from $H^1(\Omega)$ to $L^4(\Omega)$ under $ N \leq 3 $, and Young's inequality, one can see that:
\begin{align}
  &-\int_\Omega \alpha_0'(\eta^1(t)) \partial_t \eta^1(t) |(\theta^1 - \theta^2)(t)|^2 \,dx \label{dm05}
  \\
  \geq& -|\alpha_0'|_{L^\infty(-M_0, M_0)} |\partial_t \eta^1(t)|_H |(\theta_1 - \theta^2)(t)|_{L^4(\Omega)}^2 \notag
  \\
  \geq& -(C_{H^1}^{L^4})^2 |\alpha_0'|_{L^\infty(-M_0, M_0)} |\partial_t \eta^1(t)|_H |(\theta_1 - \theta^2)(t)|_V^2 \notag
  \\
  \geq& -\frac{(C_{H^1}^{L^4})^2|\alpha_0'|_{L^\infty(-M_0,M_0)}}{\delta_*(\nu)}|\partial_t \eta^1(t)|_H \cdot \notag
  \\
  & \qquad \cdot \bigl( |\sqrt{\alpha_0(\eta^1(t))}(\theta^1 - \theta^2)(t)|_H^2 + \nu^2 |\nabla (\theta^1 - \theta^2)(t)|_{[H]^N}^2 \bigr), \notag
\end{align}
and
\begin{align}
  &\int_\Omega (\alpha_0(\eta^1(t)) - \alpha_0(\eta^2(t)))\partial_t \theta^2(t) (\theta^1 - \theta^2)(t) \,dx \label{dm06}
  \\
  \geq& -|\alpha_0'|_{L^\infty(-M_0,M_0)} |\partial_t \theta^2(t)|_{L^4(\Omega)} |(\eta^1 - \eta^2)(t)|_H |(\theta^1 - \theta^2)(t)|_{L^4(\Omega)} \notag
  \\
  \geq& -(C_{H^1}^{L^4})^2 |\alpha_0'|_{L^\infty(-M_0,M_0)} |\partial_t \theta^2(t)|_V |(\eta^1 - \eta^2)(t)|_H |(\theta^1 - \theta^2)(t)|_V \notag
  \\
  \geq& -\frac{(C_{H^1}^{L^4})^2 |\alpha_0'|_{L^\infty(-M_0,M_0)} }{2}|\partial_t \theta^2(t)|_V \bigl( |(\eta^1 - \eta^2)(t)|_H^2 + |(\theta^1 - \theta^2)(t)|_V^2 \bigr) \notag
  \\
  \geq & -\frac{(C_{H^1}^{L^4})^2|\alpha_0'|_{L^\infty(-M_0,M_0)}}{2\delta_*(\nu)}|\partial_t \theta^2(t)|_V \cdot \notag
  \\
  &\quad \cdot \bigl( |(\eta^1 - \eta^2)(t)|_H^2 + |\sqrt{\alpha_0(\eta^1(t))}(\theta^1 - \theta^2)(t)|_H^2 + \nu^2|\nabla (\theta^1 - \theta^2)(t)|_{[H]^N}^2 \bigr), \notag
  \\
  &\qquad\qquad\qquad\qquad\qquad \mbox{ for a.e. } t > 0, \notag
\end{align}
where $C_{H^1}^{L^4}$ is a constant of the continuous embedding from $H^1(\Omega)$ to $L^4(\Omega)$.

Therefore, putting
\begin{align}
  J_1(t) &:= |(u^1 - u^2)(t)|_H^2 + |(v^1 - v^2)(t)|_H^2, \ \mbox{ for } t \geq 0, \notag
\end{align}
and
\begin{equation*}
  \widetilde C_{10} := 2 \bigl( |\alpha'|_{L^\infty(-M_0,M_0)} + |g'|_{L^\infty(-M_0,M_0)} + (C_{H^1}^{L^4})^2 |\alpha_0'|_{L^\infty(-M_0,M_0)} + 1 \bigr),
\end{equation*}
it is deduced from \eqref{uni_1}--\eqref{dm06} that:
\begin{equation}
  \frac{d}{dt} J(t) \leq \frac{\widetilde C_{10}}{\delta_*(\nu)} \bigl( |\partial_t \eta^1(t)|_H + |\partial_t \theta^2(t)|_V + 1 \bigr) J(t) + J_1(t), \ \mbox{ a.e. } t > 0, \label{uni_4}
\end{equation}
Applying Gronwall's lemma in \eqref{uni_4}, it can be obtained that for any $T > 0$,:
\begin{align*}
  J(t) &\leq C_1(\nu) \left( J(0) + \int_0^T J_1(s)\,ds \right), \ \mbox{ for any } t \in [0,T].
\end{align*}
with
\begin{equation*}
  C_1(\nu) := \exp\left( \frac{\widetilde C_{10}\sqrt{T}}{\delta_*(\nu)} \bigl( |\partial_t \eta^1|_\sH + |\partial_t \theta^2|_{\sV} + \sqrt{T} \bigr) \right)
\end{equation*}

Thus, we finish the proof of \cref{mainThm2}.
\qed

%\bibliography{sinst41}

\end{document}